\documentclass[
	english
]{scrartcl}

\usepackage[T1]{fontenc}
\usepackage[utf8]{inputenc}

\usepackage{amsmath,amsfonts,amsthm,amssymb}
\usepackage{xcolor}

\PassOptionsToPackage{hyphens}{url}
\usepackage[colorlinks,citecolor=blue,linkcolor=blue]{hyperref}

\usepackage{preamble}
\usepackage{marginnote}

\usepackage{enumitem}
\setlist{itemsep=0em} 
\setlist[enumerate]{label=(\roman*)}

\usepackage[capitalise,nameinlink]{cleveref}

\usepackage{soul,cancel}

\usepackage{graphicx}
\usepackage{caption}
\usepackage{subcaption}

\usepackage{microtype} 
\usepackage{lmodern}   

\usepackage{pgfplots}

\newif\ifbiber
\bibertrue
\biberfalse

\ifbiber
\usepackage[
	style=authoryear-comp,
	sorting=nyt,
	url=true,
	doi=true,                
	eprint=true,
	giveninits=true,         
	uniquename=allinit,        
	maxbibnames=99,          
	maxcitenames=2,
	uniquelist=minyear,
	dashed=false,            
	sortcites=false,          
	backend=biber,           
	safeinputenc,            
]{biblatex}

\DeclareCiteCommand{\cite}{%
	\ifbibmacroundef{cite:init}{}{\usebibmacro{cite:init}}\usebibmacro{prenote}%
}{%
	\usebibmacro{citeindex}%
	\printtext[bibhyperref]{\usebibmacro{cite}}%
}{%
	\ifbibmacroundef{cite:init}{\multicitedelim}{}%
}{%
	\usebibmacro{postnote}%
}%
\DeclareCiteCommand{\parencite}[\mkbibbrackets]{%
	\ifbibmacroundef{cite:init}{}{\usebibmacro{cite:init}}\usebibmacro{prenote}%
}{%
	\usebibmacro{citeindex}%
	\printtext[bibhyperref]{\usebibmacro{cite}}%
}{%
	\ifbibmacroundef{cite:init}{\multicitedelim}{}%
}{%
	\usebibmacro{postnote}%
}%

\let\cite\parencite

\DeclareNameAlias{sortname}{last-first}
\DeclareNameAlias{default}{last-first}

\else

\usepackage{doi}

\fi

\newcommand\N{\mathbb{N}}

\newcommand\R{\mathbb{R}}

\newcommand\MM{\mathcal{M}}

\newcommand{\eps}{\varepsilon}

\renewcommand\d{\mathrm{d}}
\newcommand\dx{\d x}
\newcommand\ds{\d s}
\newcommand\dt{\d t}

\newcommand{\weakly}{\rightharpoonup}
\newcommand{\weaklystar}{\stackrel\star\rightharpoonup}

\newcommand\dualspace{^\star}

\DeclareMathOperator*{\esssup}{ess\,sup}
\DeclareMathOperator*{\essinf}{ess\,inf}
\DeclareMathOperator*{\argmax}{arg\,max}

\DeclareMathAlphabet{\mathpzc}{OT1}{pzc}{m}{it}

\newcommand{\lse}{\mathrm{LSE}}
\newcommand{\lie}{\mathrm{LIE}}

\newcommand{\supp}{\operatorname{supp}}

\newtheorem{theorem}{Theorem}[section]
\newtheorem{lemma}[theorem]{Lemma}

\newtheorem{assumption}[theorem]{Assumption}

\newtheorem{corollary}[theorem]{Corollary}
\newtheorem{remark}[theorem]{Remark}
\newtheorem{definition}[theorem]{Definition}

\crefname{assumption}{Assumption}{Assumptions}

\graphicspath{{Plots_piecew_lin/}{Plots_sin/}}

\definecolor{darkgreen}{rgb}{0,0.5,0}
\definecolor{darkred}{rgb}{0.8,0,0}

\usepackage[normalem]{ulem}
\newcommand{\deleted}[1]{}
\newcommand{\added}[1]{#1}

\newcommand{\trp}{\text{\sffamily\slshape T}}
\newcommand{\diag}[1]{\operatorname{diag}\left(#1\right)}

\begin{document}
\title{Optimal control of ODEs with state suprema}
\author{%
	Tobias Geiger%
	\footnote{%
	Institut f\"ur Mathematik,
Universit\"at W\"urzburg,
97074 W\"urzburg, Germany, \email{tobias.geiger@mathematik.uni-wuerzburg.de}
and \email{daniel.wachsmuth@mathematik.uni-wuerzburg.de}
	}
	\and
	Daniel Wachsmuth%
  \footnotemark[1]
	\and
	Gerd Wachsmuth%
	\footnote{%
		Brandenburgische Technische Universität Cottbus-Senftenberg,
		Institute of Mathematics,
		Chair of Optimal Control,
		03046 Cottbus,
		Germany,
		\url{https://www.b-tu.de/fg-optimale-steuerung},
		\email{gerd.wachsmuth@b-tu.de}%
	}%
	\orcid{0000-0002-3098-1503}%
}
\publishers{}
\maketitle
\begin{abstract}
 We consider the optimal control of a differential equation
 that involves the suprema of the state over some part of the history.
 In many applications, this non-smooth functional dependence
 is crucial for the successful modeling of real-world phenomena.
 We prove the existence of solutions and show that related
 problems may not possess optimal controls.
 Due to the non-smoothness in the state equation,
 we cannot obtain optimality conditions via standard theory.
 Therefore, we regularize the problem
 via a LogIntExp functional which generalizes the well-known LogSumExp.
 By passing to the limit with the regularization,
 we obtain an optimality system for the original problem.
 The theory is illustrated by some
 numerical experiments.
\end{abstract}


\begin{msc}
 \mscLink{49K21},
 \mscLink{34K35},
 \mscLink{49J21}
\end{msc}

\section{Introduction}
\label{sec:introduction}

In this paper, we study optimal control problems with differential equations
that involve the suprema of the state. To be precise,
the state equation is given by
\begin{equation}\label{eq001}
 x'(t) = F\parens[\Big]{x(t), \max_{s \in [t-\tau,t]} x(s), u(t) },
 \qquad t \in (0,T)
\end{equation}
with initial data
\begin{equation}\label{eq002}
 x(t) = \phi(t), \ t\in [-\tau,0].
\end{equation}
Here, $T > 0$ is the final time \added{and} $\tau>0$ is the parameter for the state suprema.
The control $u$ has to be chosen optimal subject to an objective functional of Lagrange type
\begin{equation}
  \label{eq:objective}
  J(x,u) = \int_0^T j(t, x(t), u(t)) \, \dt
\end{equation}
and the control has to obey the constraints
\begin{equation}
 \label{eq:controlconstraint}
 u(t) \in U,
 \qquad t \in (0,T).
\end{equation}
For the precise assumptions, we refer to \cref{subsec:standing_asm} below.

Control of such systems \added{is} proposed by
\cite{AzhmyakovAhmedVerriest16,VerriestAzhmyakov2017}.
Differential equations with state suprema have an abundance of applications, we name
only solar power plant control \cite{AzhmyakovAhmedVerriest16}, population genetics \cite{Wu1996}, and quantum chemistry \cite{BainovHristova2011}.
Input-to-output stability of such systems \added{is} addressed in \cite{DASHKOVSKIY201712925}.
Systems with maximum are also popular in machine learning \cite{Goodfellow2016}, by means of the so-called max-pooling.

The evolution system \eqref{eq001} is inherently non-smooth due to the appearance of the $\max$-term.
This makes the development of necessary optimality conditions challenging. While the celebrated
Pontryagin maximum principle can account for non-smooth functions of the control $u$, this is not the case
for \added{systems} with non-smooth functions applied to the state $x$.
Therefore,
the main contribution of the present work
is the extension of \cite{Banks69}
in which a Pontryagin maximum principle is derived
for systems with a smooth functional dependence.
In \cite{ClarkeWolenski1996},
the authors studied the optimal control of a nonsmooth functional differential inclusion.
After some tedious but straightforward transformations,
the control problem \eqref{eq001}--\eqref{eq:controlconstraint} can be put in their framework
and they obtain an optimality system similar to our \cref{thm:opt_sys}.
However, our regularization approach allows for a numerical realization
and we also expect that our strategy for deriving optimality conditions
can be transferred to control problems of partial differential equations
involving state suprema.

The paper is structured as follows.
In \cref{sec:prelim}
we define some notation and
fix the standing assumption (\cref{asm:standing}).
The existence of solutions of the state equation \eqref{eq001}
and of the optimal control problem
is proven in \cref{sec:existence}.
Moreover, we show in \cref{subsec:nonexistence}
that optimal control problems
\added{
as in \cite{AzhmyakovAhmedVerriest16}
}
may lack optimal solutions.
In order to define a regularization of the state equation,
we define a LogIntExp regularization
by generalizing the well known LogSumExp function,
see \cref{sec:logintexp}.
In \cref{sec:reg},
we study the regularized problems
and we pass to the limit with the regularization
in \cref{sec:limit}.
With this technique, we arrive at the optimality system in
\cref{thm:opt_sys}
and this allows us to derive
jump conditions for the adjoint state
in
\cref{lem:jumps}.
Finally, we present some numerical examples in \cref{sec:numerics}.

\section{Notations, preliminaries and standing assumptions}
\label{sec:prelim}
\subsection{Notation}
Let us define abbreviations for time intervals
\[
 I:=[0, T], \quad I_\tau:=[-\tau,T].
\]
For a time $t \in I$ and a function $x \in C(I_\tau; \R^{\added{n}})$,
we define $x_t \in C([0,\tau]; \R^{\added{n}})$ via
\begin{equation*}
 x_t(s)
 :=
 x(t - s)
 \qquad\forall s \in [0,\tau].
\end{equation*}
This notation implies
\begin{equation*}
 \max_{s \in [t-\tau,t]} x(s)
 =
 \max_{s \in [0,\tau]} x_t(s)
 =
 \max x_t.
\end{equation*}
Note that this maximum is evaluated component-wise.
Let us define
\[
 C_\tau:= C([0,\tau];\R^n),
\]
which is \added{equipped} with the max-norm.

We will frequently use scalar functions applied to vectors. For vectors $v,w\in \R^n$, we denote by
$\exp v$ and $\frac vw$ the component-wise exponentiation and division. Moreover, $v\odot w$
denotes the Hadamard (or component-wise) product.

\subsection{Preliminaries}

Due to the retarded structure of the state equation,
we need a special integral inequality,
which is very similar to Gronwall's lemma.
\begin{lemma}
 \label{lem:max_gronwall}
 Let $x \in C(I_\tau; \R)$ be a continuous function
 with $x \equiv 0$ on $[-\tau,0]$.
 Suppose that
 there exist constants $k_1, k_2\ge 0$
 such that
 \begin{equation*}
  x(t) \le k_1 + k_2 \, \int_0^t x(s) + \max x_s \, \ds
 \end{equation*}
 holds for all $t \in I$.
 Then,
 \begin{equation*}
  x(t) \le k_1 \, \exp(2 \, k_2 \, t)
 \end{equation*}
 holds for all $t \in I$.
\end{lemma}
\begin{proof}
We refer to \cite[Theorem~2.1.1]{BainovHristova2011}.
\end{proof}

\subsection{Standing assumptions}
\label{subsec:standing_asm}
We fix the standing assumptions for the treatment of the optimal control problem
\eqref{eq001}--\eqref{eq:controlconstraint}.
These assumptions shall hold throughout the paper.
\begin{assumption}[Standing assumptions]\leavevmode
 \label{asm:standing}
 \begin{enumerate}
  \item
   The function $F : \R^n \times \R^n \times \R^m \to \R^n$
   is affine in its third argument, i.e.,
   \begin{equation}
    \label{eq:F_affine}
    F(x, v, u)
    =
    F_0(x, v)
    +
    F_1(x, v) \, u
   \end{equation}
   for
   globally Lipschitz continuous and
   continuously differentiable functions $F_0 : \R^n \times \R^n \to \R^n$, $F_1 : \R^n \times \R^n \to \R^{n \times m}$.
  \item
   The initial datum $\phi$ belongs to $C([-\tau,0];\R^n)$.
  \item
   The admissible set $U\subset \R^m$ is non-empty, convex and compact.
  \item
   The integrand $j : I \times \R^n \times \R^m \to \R$ \deleted{$ \cup \{+\infty\}$} is a \added{non-negative, }normal integrand and convex in its third argument $u$,
   see \cite[Chapter~VIII, Section~2.1]{EkelandTemam1999}.
   In addition, it is differentiable w.r.t.\ the second argument $x$
   \added{and this derivative $j_x$ is continuous w.r.t.\ the second and third arguments}.
 \end{enumerate}
\end{assumption}

Under these assumptions, a function $x\in C(I_\tau;\R^n) \cap W^{1,\infty}(I;\R^n)$ is a solution of \eqref{eq001}
if and only if it satisfies the integral equation
\begin{equation}\label{eq:integraleq}
 x(t)  = \phi(0) + \int_0^t F_0(x(s),\max x_s)  +    F_1(x(s), \max x_s) \, u(s)\, \ds \quad \forall t\in I
\end{equation}
together with the initial condition \eqref{eq002}.

\begin{lemma}\label{lem:Jliminf}
 If $x_k\to x$ in $L^1(I;\R^n)$ and $u_k\rightharpoonup u$ in $L^1(I;\R^m)$
 with $u_k(t) \in U$ f.a.a.\ $t \in I$
 then
 \[
  J(x,u) \le \liminf_{k \to \infty} J(x_k,u_k)
  .
 \]
\end{lemma}
\begin{proof}
 This follows from \cite[Chapter~VIII, Theorem~2.1]{EkelandTemam1999}.
 \added{
 Note that the growth condition
 \cite[Chapter~VIII, (2.2)]{EkelandTemam1999}
 can be fulfilled due to the boundedness of $U$:
 Let $M>0$ such that $|v|\le M$ for all  $v\in U$. Then the growth condition is
 satisfied with
 \[
 \Phi(t) = I_{[0,M]}(t) = \begin{cases} 0 & \text{ if } t\in [0,M],\\ +\infty & \text{ if } t>M,\end{cases}
 \]
 as this function is coercive, convex, and increasing on $[0,+\infty)$.
 }
\end{proof}

\begin{remark}
 Under standard modifications, the results of the paper are true for non-autonomous $F$, i.e., where $F$ is given by
 $ F(t,x, v, u)
    =
    F_0(t,x, v)
    +
    F_1(t,x, v) \, u$.
\end{remark}

\section{Existence of solutions}
\label{sec:existence}

In this section, we study the properties of the following differential equation
\begin{equation}
 \label{eq:ode}
 x'(t) = F\parens[\Big]{x(t), \max_{s \in [t-\tau,t]} x(s), u(t) },
 \qquad t \in I
\end{equation}
and the associated control problem.

\subsection{Study of a nonlinear differential equation}

Let us first study an equation more general than \eqref{eq:ode}.
We will investigate the solvability of
\begin{equation}
 \label{eq:ode_general}
 x'(t) = f(t,x(t), x_t, u(t) ),
 \qquad t \in I.
\end{equation}
subject to the initial conditions as above.
A function $x\in C(I_\tau;\R^n) \cap W^{1,\infty}(I;\R^n)$ is called a solution \deleted{of}
if \eqref{eq:ode_general} holds for almost all $t$ and the initial condition $x(t) =\phi(t)$ for all $t\in[-\tau,0]$ is satisfied.

In order to prove solvability of \eqref{eq:ode_general} for all feasible controls $u$,
we require the following assumption.

\begin{assumption}\label{ass_f}
\begin{enumerate}
 \item The function $f:I\times \R^n\times C_\tau \times \R^m \to \R^n$ is measurable in the first, and continuous with respect to the
 other arguments.
 \item For all $M>0$ exists $L_M>0$ such that
 \[
  \abs{f(t,x_1,y_1,u) - f(t,x_2,y_2,u)} \le L_M(\abs{x_1-x_2}+ \norm{y_1-y_2}_{C_\tau})
 \]
 for all $x_1,x_2\in\R^n$, $y_1,y_2\in C_\tau$, $u\in \R^m$ with $\abs{u} \le M$, and almost all $t\in I$.
 \item For all $M>0$ exists $K_M>0$ such that
 \[
  \abs{f(t,0,0,u)} \le K_M
 \]
 \added{for all} $u\in \R^m$ with $\abs{u}\le M$ and almost all $t\in I$.
\item The initial data satisfies $\phi \in C([-\tau,0];\R^n)$.
\end{enumerate}
\end{assumption}

Since $x_t \mapsto \max_s{x_t(s)}$ is Lipschitz continuous, the original problem is
covered by these assumptions. In addition, smooth regularizations of the max-functions
are included as well.

\begin{theorem}
\label{thm:existence_general}
Let \cref{ass_f} be satisfied. Let $u\in L^\infty(I;\R^m)$ be given.
Then there exists a unique solution $x\in W^{1,\infty}(I;\R^n)$ of \eqref{eq:ode_general}
\added{
with initial condition \eqref{eq002}.
}

In addition, the mapping $u \mapsto x$ maps bounded sets in $L^\infty(I;\R^m)$ to bounded sets in $W^{1,\infty}(I;\R^n)$.
\added{
That is, for all $M>0$, there exists $C_M>0$ such that
\[
 \|x\|_{W^{1,\infty}(I;\R^n)} \le C_M
\]
for all solutions $x$ of \eqref{eq:ode_general} to $u\in L^\infty(I;\R^m)$ with $\|u\|_{L^\infty(I;\R^m)}\le M$.
Here, $C_M$ depends on  $L_M,K_M$ but not on $f$.
}
\end{theorem}
\begin{proof}
We follow the proof of \cite[Thm.\ 2.1.1]{Wu1996}, which  uses a standard Picard-Lindel\"of argument.
Let us define the functions $x_k\in C(I_\tau;\R^n)$, $k=0,1,\dots$, by
\[
 x_0(t) = \phi( \min(t,0) ),
 \qquad t \in I_\tau
\]
and for $k\ge 0$
\[
 x_{k+1}(t) =
 \begin{cases} \phi(t), & \text{if } t\in [-\tau,0],\\
           \phi(0) + \int_0^t f(s,x_k(s),x_{k,s},u(s))\ds,& \text{if } t \in (0,T].
          \end{cases}
\]
In the following, $t$ is taken from $I$.
Let $M:=\norm{u}_{L^\infty(I;\R^m)}$.
By Lipschitz continuity of $f$, we obtain
\[
 \begin{aligned}
\abs{x_{k+1}(t)-x_k(t)} & \le L_M \int_0^t \abs{x_k(s)-x_{k-1}(s)} + \|x_{k,s}-x_{k-1,s}\|_{C_\tau} \ds\\
&\le 2L_M \int_0^t \|x_k-x_{k-1}\|_{C([0,s];\R^n)}\ds,
\end{aligned}
\]
which implies the estimate
\begin{equation}\label{eq107}
\|x_{k+1}-x_k\|_{C([0,t];\R^n)} = \max_{s\in [0,t]} \abs{x_{k+1}(\added{s})-x_k(s)}
\le 2L_M \int_0^t \|x_k-x_{k-1}\|_{C([0,s];\R^n)}\ds
\added{.}
\end{equation}
Due to the assumptions on $f$, we obtain
\[\begin{split}
 \abs*{\int_0^t f(s,x_0(s),x_{0,s},u(s))\ds} &\le \int_0^t L_M \, ( \abs{\phi(0)} + \norm{x_{0,s}}_{C_\tau}) + K_M \, \ds\\
  &\le (2 \, L_M \, \|\phi\|_{C([-\tau,0];\R^n)} + K_M) \, t.
\end{split}\]
By definition of $x_0$ and $x_1$, this gives the estimate
\[
 \abs{x_1(t)-x_0(t)} \le (2 \, L_M \, \|\phi\|_{C([-\tau,0];\R^n)} + K_M) \, t,
\]
which implies
\begin{equation}\label{eq108}
 \|x_1-x_0\|_{C([0,t];\R^n)} \le  (2 \, L_M \, \|\phi\|_{C([-\tau,0];\R^n)} + K_M) \, t =: K\,t.
\end{equation}
By an induction argument based on \eqref{eq107} and \eqref{eq108}, we obtain the estimate
\[
  \|x_k-x_{k-1}\|_{C([0,t];\R^n)} \le (2L_M)^{k-1} K \frac{t^k}{k!},
\]
which implies
\begin{equation}\label{eq109}
 \norm{x_k-x_{k-1}}_{C(I;\R^n)} \le (2L_M)^{k-1} K \ \frac{T^k}{k!}.
\end{equation}
Since $\sum_{k=1}^\infty (2L_M)^{k-1} K T^k \ \frac1{k!}< \infty$, the sequence $(x_k)$ is a Cauchy sequence
in $C(I;\R^n)$, hence convergent to some $x\in C(I;\R^n)$.
It remains to show that $x$ solves \eqref{eq:ode_general}.
Let us estimate
\begin{multline*}
 \abs*{x(t) - \phi(0) - \int_0^t f(s,x(s),x_s,u(s))\ds}\\
 \begin{aligned}
 &\le \abs{x(t)- x_{k+1}(t)} + \int_0^t \abs{f(s,x_k(s),x_{k,s},u(s))-f(s,x(s),x_s,u(s))}\ds\\
 & \le \|x-x_{k+1}\|_{C(I;\R^n)}+2 L_M T \|x-x_k\|_{C(I;\R^n)}.
 \end{aligned}
\end{multline*}
Passing to the limit $k\to\infty$,
we find that $x$ solves the integral equation
\[
 x(t) = \phi(0) + \int_0^t f(s,x(s),x_{s},u(s))\ds,
\]
which implies that the weak derivative of $x$ satisfies
\[
 x'(t) = f(t,x(t),x_t,u(t))
\]
for almost all $t$. This right-hand side is in $L^\infty(I;\R^n)$, hence $x\in W^{1,\infty}(I;\R^n)$ is proven.

Let $x_1,x_2 \in W^{1,\infty}(I;\R^n)$ be two solutions. Then it holds
\[
 x_1(t)-x_2(t) = \int_0^t f(s,x_1(s),x_{1,s},u(s)) - f(s,x_2(s),x_{2,s},u(s)) \, \ds.
\]
Arguing similarly as in the derivation of \eqref{eq107} above, we find
\[
\|x_1-x_2\|_{C([0,t];\R^n)}
\le 2L_M \int_0^t \|x_1-x_2\|_{C([0,s];\R^n)}\ds,
\]
which implies $x_1=x_2$ by Gronwall's lemma.

Let us prove the claimed boundedness result. By the construction of $x_0$ and the definition of $K$, we get $\norm{x_0}_{C(I;\R^n)}\le K$.
Summing inequality \eqref{eq109}, gives
\[
 \norm{x}_{C(I;\R^n)} \le K + \sum_{k=1}^\infty (2L_M)^{k-1} K T^k \ \frac1{k!}
 \le \max(1, (2L_M)^{-1}) K e^{2L_MT}.
\]
In addition, we have
\[
 \norm{x'}_{L^\infty(I;\R^n)} =\norm{f(\cdot,x,x_t,u)}_{L^\infty(I;\R^n)}\le 2\,L_M \norm{x}_{C(I;\R^n)} + K_M.
\]

Let now $\tilde U \subset L^\infty(I;\R^m)$ be a bounded set with $\|u\|_{L^\infty(I;\R^m)}\le M$ for all $u\in \tilde U$.
Then the estimate above is uniform with respect to controls $u\in \tilde U$, which proves the claim.
\end{proof}

Let us mention that the proof implies that the mapping $\phi\mapsto x$ is Lipschitz continuous.

\begin{corollary}
 \label{cor:existence_state_equation}
Let \cref{asm:standing} be satisfied. Let $u\in L^\infty(I;\R^m)$ be given.
Then there exists a unique solution $x\in W^{1,\infty}(I;\R^n)$ of \eqref{eq001}--\eqref{eq002}.

In addition, the mapping $u \mapsto x$ maps bounded sets in $L^\infty(I;\R^m)$ to bounded sets in $W^{1,\infty}(I;\R^n)$.
\end{corollary}
\begin{proof}
 The result follows directly from \cref{thm:existence_general}, as \cref{asm:standing} implies \cref{ass_f}.
\end{proof}

\subsection{Existence of solutions of the optimal control problem}

\begin{theorem}
\label{thm:existence_optimal_control}
 The optimal control problem \eqref{eq001}--\eqref{eq:controlconstraint} admits \added{at least} one solution.
\end{theorem}
\begin{proof}
Due to the assumptions and \cref{cor:existence_state_equation}, the feasible set is non-empty.
Let $(x_k,u_k)$ be a minimizing sequence. By assumptions on $U$, the sequence $(u_k)$ is bounded in $L^\infty(I;\R^n)$, and by \cref{thm:existence_general}
the sequence $(x_k)$ is bounded in $W^{1,\infty}(I;\R^n)$.
Hence, after possibly extracting subsequences, we have $u_k\rightharpoonup u$ in $L^2(I;\R^n)$ and $x_k\to x$ in $C(I_\tau;\R^n)$.
Since the set of admissible controls is convex and closed in $L^2(I;\R^n)$, the feasibility of $u$ follows.
Due to the special structure of $F$, we can pass to the limit in the integral equation \eqref{eq:integraleq}.
Hence, $x$ solves the integral equation, and consequently is a solution of \eqref{eq001}.
Due to \cref{lem:Jliminf}, we obtain that $(x,u)$ is a solution of the considered optimal control problem.
\end{proof}

\subsection{Non-existence of optimal controls for related problems}
\label{subsec:nonexistence}

\added{
In \cite{AzhmyakovAhmedVerriest16}, optimal control problems were studied that do not satisfy the \cref{asm:standing}, in particular
the structural assumption on $F$ is violated. They considered state equations of the type
\begin{equation}\label{eq:uinmax}
 x'(t) = Ax(t) + b k \max_{s \in [t-\tau,t]}\left( c(s)^\trp x(s)\right), \quad t \in I,
\end{equation}
where $A\in \R^{n\times n}$, $b\in \R^n$ are given, and $k\in \R$ and $c(t)\in \R^n$ are control variables.
The control variable $c(t)$ appears in the argument of the max-function.
Hence, the right-hand side of this equation depends in a highly nonlinear way on the control $c$.
}

\added{
In this section, we are going to demonstrate that an optimal control problem subject to an equation of the type \eqref{eq:uinmax}
may fail to possess optimal solutions. To this end,
}
we consider the problem with one-dimensional state and control
\begin{equation}
 \label{eq:non_existence}
 \begin{aligned}
  \text{Minimize} &\quad \int_0^1 \abs{x(t) - 2\,t}^2 + u(t) \, \dt
  + 4 \, \abs{x(1) - 2} \\
  \text{s.t.} &\quad x'(t) = \frac1{10} \, \max_{s \in [t-2,t]}( u(t) \, x(s) ), \qquad t \in (0,1)\\
  & \quad x(t) = \phi(t), \qquad t \in [-2,0] \\
  & -1 \le u \le 3.
 \end{aligned}
\end{equation}
Moreover,
$\phi : [-2,0] \to \R$ is a function
such that the ranges of $\phi_{|[-2,-1]}$, $\phi_{|[-1,0]}$ are $[-10,10]$
and $\phi(0) = 0$.

Before \deleted{we} discussing the existence of controls,
we are going to analyze the state equation.
By splitting the $\max$ at $s = 0$,
we find
\begin{equation*}
 \abs{x'(t)} \le 3 + \frac3{10} \max_{s \in [0,t]} \abs{x(s)}.
\end{equation*}
Integration yields
\begin{equation*}
 \abs{x(t)} \le 3 \, t + \frac3{10} \int_0^t \max_{s \in [0,r]} \abs{x(s)} \, \d r.
\end{equation*}
Since this estimate is valid for all $t$
and since the right-hand side is monotone w.r.t. $t$,
this implies
\begin{equation*}
 \max_{s \in [0,1]} \abs{x(s)}
 \le
 3 + \frac3{10} \int_0^1 \max_{s \in [0,r]} \abs{x(s)} \, \d r
 \le
 3 + \frac3{10} \max_{s \in [0,1]} \abs{x(s)}
 .
\end{equation*}
Thus,
$\abs{x(t)} \le 30/7 \le 10$.
This estimate allows us to evaluate the $\max$ in the state equation
via
\begin{equation*}
 \max_{s \in [t-2,t]} (u(t) \, x(s) )
 =
 \max_{s \in [-1,0]} (u(t) \, x(s) )
 =
 10 \, \abs{u(t)}.
\end{equation*}
Hence, we can simplify the state equation and obtain
\begin{equation*}
 x'(t) = \abs{u(t)}.
\end{equation*}
Now, we can prove the main result of this section.
\begin{theorem}
 \label{thm:nonexistence}
 The control problem \eqref{eq:non_existence}
 does not possess a solution.
\end{theorem}
\begin{proof}
 \textit{Step 1:}
 We show that the infimal value is at most $1$.
 Let us define
 $u_k(t) = 1 + 2 \, \operatorname{sign}(\sin(k \, t))$
 and let $x_k$ be the associated state.
 Then, it is straightforward to check that
 $u_k \weakly \hat u \equiv 1$ in $L^2(0,1)$
 and
 $x_k \to \hat x$
 \added{in $C(I_\tau; \R)$}
 with $\hat x(t) = 2 \, t$.
 This implies that the objective value of $(x_k, u_k)$
 goes to $0 + 1 + 0 = 1$.
 Note that $\hat x$ is not the state corresponding to $\hat u$.

 \textit{Step 2:}
 We show that the objective value of all feasible $(x,t)$
 is bigger than $1$.
 We proceed by contradiction and
 assume that we have a feasible pair
 $(x,u)$ with objective value at most $1$.

 Let us denote
 \begin{equation*}
  a := \int_{\{u < 0\}} u \, \dt,
  \qquad
  b := \int_{\{u > 0\}} u \, \dt.
 \end{equation*}
By considering the control bounds and the length of the time interval, we obtain
 \[
 1 =\int_{\{u < 0\}} 1 \, \dt + \int_{\{u > 0\}} 1 \, \dt
 \ge \int_{\{u < 0\}} (-u) \, \dt + \int_{\{u > 0\}} \frac13u\, \dt
 =  -a + \frac b3 .
  \]
 This inequality is equivalent to
 \begin{equation}
  \label{eq:bound_ab}
  -1 \ge 2(b-a-2) -(a+b).
\end{equation}
 Since the objective is at most $1$,
 we have
 \begin{equation}
  \label{eq:bound_ab2}
  1 \ge
  \int_0^1 u \, \dx + 4 \, \abs{ x(1) - 2}
  =
  \int_0^1 u \, \dt + 4 \, \abs*{ \int_0^1\abs{u}\,\dt - 2}
  =
  a + b + 4 \, \abs{ b - a - 2 }.
 \end{equation}
 Adding the inequalities \eqref{eq:bound_ab} and \eqref{eq:bound_ab2} yields $b-a-2 = 0$, which in turn implies $\int_0^1 u \,\dt = a+b=1$.
 By considering again the objective,
 we infer
 $x(t) = 2 \, t$, i.e., $\abs{u(t)} = x'(t) = 2$.
 Hence, $u(t) = 2$ and this is a contradiction.
\end{proof}
\deleted{
As a side result,
this theorem also shows that it is not possible
to relax the assumption that $F$ is affine in $u$, cf.\ \eqref{eq:F_affine}.
}

\section{LogIntExp as a generalization of LogSumExp}
\label{sec:logintexp}

Let $x_i$, $i = 1,\ldots, n$ be given real numbers.
It is well known that the maximum
$\max(x_1, \ldots, x_n)$
depends in a non-smooth way on the parameters $x_i$.
This has severe drawbacks in many applications.
Therefore, a typical substitute for this hard maximum
is the so-called LogSumExp function.
For a given parameter $k > 0$,
this function is defined via
\begin{equation}
	\label{eq:lse}
	\lse_k(x_1, \ldots, x_n)
	:=
	\frac1k \, \log\parens[\bigg]{ \sum_{i = 1}^n \exp(k \, x_i) }
	.
\end{equation}
In the next lemma, we summarize some of the well-known properties of LogSumExp,
see, e.g., \cite[Example~1.30]{RockafellarWets1998},
\cite[p.~325]{Rockafellar1970}.
\begin{lemma}
	\label{lem:lse}
	Let $k > 0$ be a given parameter.
	Then, the function $\lse_k$
	is convex,
	smooth,
  and
	Lipschitz continuous with rank $1$.
	The estimate
	\begin{equation*}
		\max(x_1,\ldots,x_n)
		\le
		\lse_k(x_1,\ldots, x_n)
		\le
		\max(x_1,\ldots,x_n)
		+
		\frac{\log(n)}{k}
	\end{equation*}
	shows that $\lse_k(x_1,\ldots, x_n) \to \max(x_1, \ldots, x_n)$ as $k \to \infty$.
	Concerning the derivatives, we have
	\begin{equation*}
		\frac{\d}{\d x_i} \lse_k(x_1, \ldots, x_n)
		=
		\frac{\exp(k \, x_i)}{\sum_{j = 1}^n \exp(k \, x_j)}
		.
	\end{equation*}
	In particular,
	\begin{equation*}
		\frac{\d}{\d x_i} \lse_k(x_1, \ldots, x_n)
		\ge 0
		\qquad\text{and}\qquad
		\sum_{i = 1}^n \frac{\d}{\d x_i} \lse_k(x_1, \ldots, x_n) = 1.
	\end{equation*}
\end{lemma}
Due to these nice properties, the LogSumExp function
has
many applications, e.g., in machine learning \cite{Goodfellow2016}.

A natural extension of this idea to smooth the
essential supremum of measurable functions is to replace summation with integration.
To this end, we consider a measure space
$(\Omega, \Sigma, \mu)$ with $\mu(\Omega) < \infty$.
For convenience, integrals w.r.t.\ $\mu$ are indicated by
$\int_\Omega \ldots \dx$.
Let $u : \Omega \to \R$ be a measurable function.
We define the LogIntExp of $u$ with parameter $k > 0$
via
\begin{equation}
	\label{eq:lie}
	\lie_k( u )
	:=
	\frac1k \, \log\parens[\bigg]{\int_\Omega \exp(k \, u(x)) \, \dx}
	.
\end{equation}

In \cite{KruseUlbrich2015}, the LogIntExp-function (up to some normalization factor in the argument of the $\log$)
was used for regularization of state constraints in an optimal control problem, see \cite[Definition~3.5]{KruseUlbrich2015}.
There, this function was analyzed on spaces of (H\"older) continuous functions.

Let us first state some basic properties of
LogIntExp.
\begin{lemma}
	\label{lem:logintexp_1}
	For every measurable $u : \Omega \to \R$ and $k > 0$,
	the quantity $\lie_k(u) \in \R \cup \{+\infty\}$ is well-defined
  and \added{$\lie_k(\cdot)$ is} convex.
	Moreover, we have the estimates
	\begin{equation*}
		\essinf_{x \in \Omega} u(x) + \frac{\log(\mu(\Omega))}{k}
		\le
		\lie_k(u)
		\le
		\esssup_{x \in \Omega} u(x) + \frac{\log(\mu(\Omega))}{k}.
	\end{equation*}
	If for some $\delta \in \R$, the measure of the set $\{x \in \Omega \mid u(x) \ge \delta\}$
	is at least $\varepsilon > 0$,
	then
	\begin{equation*}
		\delta + \frac{\log(\varepsilon)}{k}
		\le
		\lie_k(u)
		.
	\end{equation*}
	In particular,
	\begin{equation*}
		\lie_k(u)
		\to
		\esssup_{x \in \Omega} u(x)
		\in
		\R \cup \{+\infty\}
		\qquad\text{as } k \to \infty.
	\end{equation*}
\end{lemma}
\begin{proof}
	First, we comment on the well definedness.
	The function $x \mapsto \exp(k \, u(x))$ is measurable and positive,
	therefore the Lebesgue integral
	\begin{equation*}
		\int_\Omega \exp(k \, u(x)) \, \dx \in \R \cup \{+\infty\}
	\end{equation*}
	is well defined. Taking the logarithm
	(with the convention $\log(+\infty) = +\infty$) yields the claim.

	The convexity of $\lie_k$ is a simple application of Hölder's inequality.
	Indeed, for measurable functions $u, v$
	satisfying $\lie_k(u),\lie_k(v) < \infty$
	and $\lambda \in (0,1)$,
	we have
	\begin{align*}
		\lie_k\parens[\big]{\lambda \, u + (1-\lambda) \, v}
		&=
		\frac1k \, \log\parens[\bigg]{\int_\Omega \exp\parens[\big]{\lambda \, k \, u + (1-\lambda) \, k \, v} \, \dx}
		\\
		&=
		\frac1k \, \log\parens[\bigg]{\int_\Omega \exp\parens[\big]{\lambda \, k \, u} \, \exp\parens[\big]{(1-\lambda) \, k \, v} \, \dx}
		\\
		&\le
		\frac1k \, \log\parens[\bigg]{
			\bracks[\Big]{\int_\Omega \exp\parens[\big]{\lambda \, k \, u}^{1/\lambda}}^{\lambda}
			\,
			\bracks[\Big]{\int_\Omega \exp\parens[\big]{(1-\lambda) \, k \, v}^{1/(1-\lambda)}}^{1-\lambda}
		}
		\\
		&=
		\frac\lambda k \, \log\parens[\bigg]{\int_\Omega \exp\parens[\big]{k \, u} \, \dx}
		+
		\frac{1-\lambda}k \, \log\parens[\bigg]{\int_\Omega \exp\parens[\big]{k \, v} \, \dx}
		\\
		&=
		\lambda \, \lie_k(u) + (1-\lambda) \, \lie_k(v).
	\end{align*}
	The estimates for $\lie_k(u)$ follow from
	\begin{equation*}
		\mu(\Omega) \, \exp(k \, \essinf_{x \in \Omega} u(x))
		\le
		\int_\Omega \exp(k \, u(x)) \, \dx
		\le
		\mu(\Omega) \, \exp(k \, \esssup_{x \in \Omega} u(x))
	\end{equation*}
	and
	\begin{equation*}
		\varepsilon \, \exp(k \, \delta)
		\le
		\int_\Omega \exp(k \, u(x)) \, \dx
		.
	\end{equation*}
\end{proof}

A related error estimate for LogIntExp applied to H\"older continuous functions can be found in \cite[Lemma~3.20]{KruseUlbrich2015}.
On the space $L^\infty(\Omega)$,
we have the following properties.
\begin{lemma}
	\label{lem:lie_on_linfty}
	The function $\lie_k : L^\infty(\Omega) \to \R$
	is continuously differentiable with
	\begin{equation*}
		\lie_k'(u) \, v
		=
		\frac{\int_\Omega \exp(k \, u(x)) \, v(x) \, \dx}
		{\int_\Omega \exp(k \, u(\hat x)) \, \d\hat x}
    \qquad\added{\forall v \in L^\infty(\Omega)}.
	\end{equation*}
	In particular, $\lie_k$ is Lipschitz continuous with rank $1$.
\end{lemma}
\begin{proof}
	Standard results on Nemytskii operators imply that
  $u \mapsto \exp(k \, u)$ is {continuously} Fréchet differentiable from
	$L^\infty(\Omega)$ to $L^1(\Omega)$
	with derivative
	$k \, \exp(k \, u)$.
	Hence, the chain rule implies the announced formula for the derivative of $\lie_k$.
	Finally, the function
	$x \mapsto \exp(k \, u(x)) / \int_\Omega \exp(k \, u(\hat x)) \, \d\hat x$
	belongs to $L^1(\Omega)$ and has norm $1$.
	Thus, the weak mean value theorem, see \cite[Proposition~3.3.1]{Cartan1967},
	implies the Lipschitz continuity of $\lie_k$.
\end{proof}
This Lipschitz continuity implies
that
\begin{align*}
	\abs[\big]{ \lie_k(u_k) - \esssup u }
	&\le
	\abs[\big]{ \lie_k(u_k) - \lie_k(u) }
	+
	\abs[\big]{ \lie_k(u) - \esssup u }
	\\&
	\le
	\norm{u_k - u}_{L^\infty(\Omega)}
	+
	\abs[\big]{ \lie_k(u) - \esssup u }
	\to
	0
\end{align*}
for $u_k \to u$ in $L^\infty(\Omega)$.

Finally, we provide an estimate specialized to our problem \eqref{eq:ode}.
From now on, the measure space $\Omega$ is just the interval
$[0,\tau]$ (with the one-dimensional Lebesgue measure).
Recall that $x_t(s) = x(t - s)$ for $s \in [0,\tau]$.
In particular,
\begin{equation*}
 \lie_k(x_t)
 =
 \frac1k \, \log\parens[\Big]{\int_0^\tau \exp\parens[\big]{k \, x_t(s)} \, \ds}
 =
 \frac1k \, \log\parens[\Big]{\int_{t-\tau}^t \exp\parens[\big]{k \, x(s)} \, \ds}
 .
\end{equation*}

\begin{lemma}
 \label{lem:l1norm}
 Let $x \in L^\infty(-\tau, T)$ be a given function.
 Then,
 \begin{equation*}
  \int_0^T \abs[\big]{\esssup x_t - \lie_k(x_t)}^{\added{p}} \, \dt
  \to
  0
 \end{equation*}
 \added{for all $p \in [1,\infty)$.}
\end{lemma}
\begin{proof}
 The convergence result from \cref{lem:logintexp_1}
 implies the pointwise convergence
 \begin{equation*}
  \abs[\big]{\esssup x_t - \lie_k(x_t)} \to 0
 \end{equation*}
 for all $t \in I$.
 Moreover, we have the \added{constant} bound
 \begin{equation*}
  \abs[\big]{\esssup x_t - \lie_k(x_t)}
  \le
  2\,
  \norm{x}_{L^\infty(-\tau,T)}
  +
  \frac{\abs{\log(\tau)}}{k}
  \le
  2\,
  \norm{x}_{L^\infty(-\tau,T)}
  +
  \abs{\log(\tau)}
  ,
 \end{equation*}
 cf.\ \cref{lem:logintexp_1}.
 Thus, the dominated convergence theorem yields the claim.
\end{proof}

\section{Regularization}
\label{sec:reg}

\subsection{Regularized state equation}
As a regularization of \eqref{eq:ode},
we use
\begin{equation}
 \label{eq:ode_reg}
 x'(t) = F(x(t), \lie_k(x_t), u(t) ),
 \qquad t \in I.
\end{equation}
Note that $\lie_k(x_t)$ is understood component-wise.
The existence of
a unique solution follows
directly from
\cref{thm:existence_general},
due to the Lipschitz continuity of $\lie_k$.

The convergence of the solutions of \eqref{eq:ode_reg}
towards the solution of \eqref{eq:ode}
is made precise in the next result.
Let us emphasize that the proof heavily relies on the affine structure of $F$.

\begin{lemma}
 \label{lem:regularization_compact_state}
 Let the sequence $(u_k)$ be bounded in $L^\infty(I;\R^m)$.
 Let $x_k$ be the solution of
 \begin{equation*}
  x_k'(t) = F(x_k(t), \lie_k(x_{k,t}), u_k(t) ),
  \qquad t \in I,
 \end{equation*}
 with initial data $x_k(t) = \phi(t)$ for $t \in [-\tau,0]$.
 If $u_k \weakly u$ in $L^1(I;\R^m)$
 then $x_k \to x$ in $C(I;\R^n)$,
 \added{
 where $x$ is the solution of \eqref{eq001}--\eqref{eq002} to $u$.
 }
\end{lemma}
\begin{proof}
Let $M>0$ such that $\|u_k\|_{L^\infty(I;\R^m)}\le M$, which implies $\|u\|_{L^\infty(I;\R^m)}\le M$.
Since $\lie_k$ is Lipschitz continuous uniformly with respect to $k$, the solutions $(x_k)$ are bounded in $W^{1,\infty}(I;\R^n)$
by \cref{thm:existence_general}.
Let $L>0$ denote the maximum of the Lipschitz moduli of $F_0$ and $F_1$.

 First, we investigate the difference of the equations
 \begin{align*}
  x'(t) - x_k'(t)
  &=
  F_0( x(t), \max x_t) - F_0( x_k(t), \lie_k(x_{k,t}))
  \\&\qquad
  +
  F_1( x(t), \max x_t) u(t)
  -
  F_1( x_k(t), \lie_k(x_{k,t}))u_k(t)
  \\
  &=
  F_0( x(t), \max x_t) - F_0( x_k(t), \lie_k(x_{k,t}))
  \\&\qquad
  +
  F_1( x(t), \max x_t)(u(t) - u_k(t))
  \\&\qquad
  +
  \parens[\big]{
   F_1( x(t), \max x_t)
   -
   F_1( x_k(t), \lie_k(x_{k,t}))
  }
   u_k(t)
  .
 \end{align*}
 By Lipschitz continuity, we have
 \begin{multline*}
  \abs{ F_0( x(t), \max x_t) - F_0( x_k(t), \lie_k(x_{k,t})) }\\
 \begin{aligned}
  &  \le L \left( \abs{ x(t) - x_k(t)} + \abs{ \max x_t -  \lie_k(x_{k,t}) } \right)   \\
  &  \le L \left( \abs{ x(t) - x_k(t)} + \abs{ \max x_t - \lie_k(x_t)} + \abs{\lie_k(x_t) - \lie_k(x_{k,t}) } \right)   \\
  &  \le L \left( \abs{ x(t) - x_k(t)} + \abs{ \max x_t - \lie_k(x_t)} + \max \abs{x_t - x_{k,t} } \right) .
\end{aligned}
\end{multline*}
Similarly, we can estimate
 \begin{multline*}
  \abs{
  \parens[\big]{
   F_1( x(t), \max x_t)
   -
   F_1( x_k(t), \lie_k(x_{k,t}))
  }
  u_k(t)
  }\\
  \le  LM \left( \abs{ x(t) - x_k(t)} + \abs{ \max x_t - \lie_k(x_t)} + \max \abs{x_t - x_{k,t} } \right) .
  \end{multline*}
 Integrating over $t$
 and using
 the Lipschitz estimates above, we have
 \begin{align*}
  \abs{x(t) - x_k(t)}
  &\le
  L(M+1) \, \int_0^t \abs{x(s) - x_k(s)} + \max\abs{x_s - x_{k,s}} \, \ds
  \\&\qquad
  +
  \abs*{ \int_0^t  F_1( x(s), \max x_s)(u(s) - u_k(s))   \, \ds }
  \\&\qquad
  +
  L (M+1) \, \int_0^t  \abs{\max x_s - \lie_k(x_s)} \, \ds
 \end{align*}
 for all $t \in I$.
 Since $(x_k)$ is bounded in $W^{1,\infty}(I;\R^n)$,
 there exists a subsequence (without relabeling), such that
 $x_k \to \tilde x$ in $C(I; \R^n)$.
 Passing to the limit in the above inequality
 yields
 \begin{equation*}
  \abs{x(t) - \tilde x(t)}
  \le
  L \, (M+1) \, \int_0^t \abs{x(s) - \tilde x(s)} + \max\abs{x_s - \tilde x_s} \, \ds
 \end{equation*}
 for all $t \in I$.
 The integral inequality from \cref{lem:max_gronwall} implies $x = \tilde x$.
 Thus, a standard subsequence-subsequence argument implies $x_k \to x$ in $C(I;\R^n)$
 for the entire sequence $x_k$.
\end{proof}

\begin{corollary}
 Under the assumptions of \cref{lem:regularization_compact_state}, we have $x_k\rightharpoonup x$ in $W^{1,1}(I;\R^n)$.
\end{corollary}
\begin{proof}
 Using the result of \cref{lem:regularization_compact_state} and \cref{asm:standing}, it is easy to see that the mappings $t\mapsto F(x_k(t), \lie_k(x_{k,t}), u_k(t) )$
 converge weakly in $L^1(I;\R^n)$ to  $t\mapsto F(x(t), \max x_t, u(t) )$.
\end{proof}

\begin{remark}
 If $F$ is supposed to be nonlinear with respect to the control, then the result of the above \cref{lem:regularization_compact_state}
 is no longer true. In order to obtain convergence of the states, one has to assume strong convergence $u_k\to u$.
 However, in \cref{sec52} below, we have to deal with weakly converging sequences of controls $(u_k)$, see, e.g., the proof of \cref{thm:convergence_regularized_optimal_controls}.
 Hence, we are restricted to the affine setting.
\end{remark}

\subsection{Regularized optimal control problem}
\label{sec52}

In the following, let $(x^*,u^*)$ be a local solution of the original problem.
Then there is $\delta>0$ such that
$J(x^*,u^*)\le J(x,u)$ for all feasible controls $u$ with associated states $x$ satisfying $\|u-u^*\|_{L^2(I;\R^m)}\le\delta$.

Let us consider the following regularized optimal control problem: Minimize
\begin{equation}
  \label{eq:objective-regularized}
  J(x,u) + \frac12 \, \norm{u - u^*}_{L^2(I;\R^m)} = \int_0^T j(t, x(t), u(t)) + \frac12 \, \abs{u(t)-u^*(t)}^2 \, \dt
\end{equation}
subject to the non-linear equation \eqref{eq:ode_reg}, the initial condition \eqref{eq002}, the control constraints \eqref{eq:controlconstraint}, and the auxiliary
constraints
\begin{equation}
  \label{eq:constraint-regularized}
 \|u-u^*\|_{L^2(I;\R^m)} \le \delta.
\end{equation}

\begin{theorem}
For every $k$, the regularized optimal control problem
admits a global solution $(x_k,u_k)$.
\end{theorem}
\begin{proof}
 The proof can be carried out using similar arguments as in the proof of \cref{thm:existence_optimal_control}.
\end{proof}

\begin{theorem}
\label{thm:convergence_regularized_optimal_controls}
Let $(x_k,u_k)$ be a sequence of global solutions of the regularized optimal control problem.
Then we have the convergence $x_k\to x^*$ and $u_k\to u^*$ in $C(I;\R^n)$ and $L^2(I;\R^m)$, respectively.
\end{theorem}
\begin{proof}
 Let $x^*_k$ denote the solution of the regularized differential equation \eqref{eq:ode_reg} to the control $u^*$.
 Then \cref{lem:regularization_compact_state} implies $x^*_k\to x^*$ in $C(I;\R^n)$. The integrand $j$ is continuous with respect to $x$,
 which implies $J(x^*_k,u^*)\to J(x^*,u^*)$.
 Since $(u_k)$ is a bounded sequence in $L^2(I;\R^m)$, we obtain after extracting a subsequence $u_{k_n}\rightharpoonup \bar u$
 in $L^2(I;\R^m)$.
 Then $\bar u$ satisfies the control constraints as well as $\|\bar u-u^*\|_{L^2(I;\R^m)} \le \delta$.
 Using \cref{lem:regularization_compact_state} again, we find $x_{k_n}\to \bar x$ in $C(I;\R^n)$, which solves \eqref{eq001} to $\bar u$.
 In addition, \cref{lem:Jliminf} yields $J(\bar x,\bar u)\le \liminf_{n\to\infty} J(x_{k_n},u_{k_n})$.
By global optimality, we have
\[
J(x_k,u_k) + \frac12 \|u_k-u^*\|_{L^2(I;\R^m)}^2 \le J(x^*_k,u^*).
\]
Passing to the limit along the subsequence yields
\begin{multline*}
J(\bar x,\bar u) + \frac12 \|\bar u-u^*\|_{L^2(I;\R^m)}^2
\le \liminf_{n\to \infty}\left( J(x_{k_n},u_{k_n})+ \frac12 \|u_{k_n}-u^*\|_{L^2(I;\R^m)}^2 \right)\\
\le \limsup_{n\to \infty}\left( J(x_{k_n},u_{k_n})+ \frac12 \|u_{k_n}-u^*\|_{L^2(I;\R^m)}^2 \right)
\le J(x^*,u^*),
\end{multline*}
which implies $(\bar x,\bar u) = (x^*,u^*)$.
Hence, the above chain of inequalities are equalities, which imply the strong convergence $u_{k_n}\to u^*$ in $L^2(I;\R^m)$.
Since the limit is independent of the chosen subsequence, we obtain convergence of the whole sequence.
\end{proof}

Let $(x_k,u_k)$ be locally optimal for the regularized problem. For abbreviation, let us define
\[
 F^k(t):= F(x_k(t),\lie_k(x_{k,t}),u_k(t)).
\]
Similarly, we define  $F^k_x(t)$, $j^k_x$, and $F^k_y(t)$ to be the derivatives of $F$ and $j$ with respect to the first and second argument, respectively.

\begin{theorem}
\label{thm:optimality_condition_regularized}
Let $(x_k,u_k)$ be locally optimal for the regularized problem with $\|u_k-u^*\|_{L^2(I)}<\delta$.
Then there exists $\lambda_k\in W^{1,\infty}(I;\R^n)$ with $\lambda_k(T)=0$  such that
\begin{equation}\label{eq:adjoint_regularized}
  \lambda_k'(s)^\trp +\lambda_k(s)^\trp  F^k_x(s)
+\int_{s}^{\min(s+\tau,T)} \lambda_k(t)^\trp F^k_y(t)
		\diag{
		\frac{\exp(k \, x_k(s))}
		{\int_{t-\tau}^t \exp(k \, x_k(\hat s)) \, \d\hat s}
		}
		\dt
  =j^k_x(s)
\end{equation}
is satisfied for almost all $s\in \added{I}$.
Here, $\diag v$ denotes the diagonal matrix with diagonal entries taken from the vector $v$.
The division and exponentiation has to be understood component-wise.

Moreover, the maximum principle in integrated
form holds
\begin{multline}\label{eq:maximum-principle}
 \int_0^T \lambda_k(t)^\trp F(x_k(t),\lie_k(x_{k,t}),u_k(t)) - j(t,x_k(t),u_k(t)) - (u_k(t)-u^*(t))^\trp u_k(t)\, \dt \ge
 \\
 \int_0^T \lambda_k(t)^\trp F(x_k(t),\lie_k(x_{k,t}),u(t)) - j(t,x_k(t),u(t))- (u_k(t)-u^*(t))^\trp u(t) \dt
\end{multline}
for all $u$ satisfying the control constraint \eqref{eq:controlconstraint}.
\end{theorem}
\begin{proof}
Due to the assumptions, the control $u_k$ is a local solution of an optimal control problem without the constraint \eqref{eq:constraint-regularized}.
 We are going to apply \cite[Theorem 1]{Banks69}.
 \added{%
  Note that the regular point condition therein,
  is only required in the case of a free end time
  or terminal constraints,
  both of which are absent in our problem,
  see also the discussion in \cref{rem:on_regular_point}.%
 }
 Due to the standing \cref{asm:standing}, the requirements on the problem are fulfilled.
 Hence, there exists multipliers $\lambda_0$ and $\lambda$, satisfying a system that constitutes the
 optimality conditions of the regularized problem.
 We will develop this system in the course of the proof using the notation of \cite{Banks69}.

 Since the control problem does not include constraints on $x(T)$, we can set $\lambda_0=-1$.
 By
 \added{
 the theorem mentioned above,
 }
 there exists an adjoint state $\lambda:I\to \R^n$ of bounded variation,
 such that
 \begin{equation}\label{eq523}
  \lambda(s)^\trp + \int_s^T \lambda_0 \eta_0(t,s) + \lambda(t)^\trp \eta(t,s)\dt = \lambda^\trp(T)=0 \quad \forall s\in[0,T).
 \end{equation}
Here, the matrix-valued quantities $\eta(s,t)$ and $\eta_0(s,t)$ are defined by the equations%
\footnote{
The $i$-th component of the integral on the right-hand side is defined as $\sum_{j=1}^n \int_{-\tau}^t  \xi_j(s) \d_s\eta_{i,j}(t,s)$.
}
\[
F^k_x(t) \xi(t)+ F^k_y(t)\lie_k'(x_{k,t}) \xi_t
= \int_{-\tau}^t  \d_s\eta(t,s)\xi(s)
\quad \forall \xi\in C(I_\tau;\R^n), t\in I,
\]
where the latter integral denotes the Lebesgue-Stieltjes integral with respect to the integration variable $s$,
and
\[
j^k_x(t) \xi(t) = \int_{-\tau}^t  \d_s\eta_0(t,s)\xi(s) \quad \forall \xi\in C(I_\tau;\R), t\in I.
\]
In order to investigate the adjoint equation, we have to calculate an explicit expression of $\eta$.
First, it is not difficult to check, see also \cite[Section 4]{Banks69}, that
it holds
\[
 F^k_x(t) \xi(t) = \int_{-\tau}^t  \d_s\eta_1(t,s)\xi(s)
 \quad \forall \xi\in C(I_\tau;\R^n), t\in I
\]
for $\eta_1$ defined by
\[
\eta_1(t,s) = -\chi_{[-\tau,t)}(s)\cdot   F^k_x(t),
\]
which directly implies
\begin{equation}\label{eq:eta1}
\int_s^T \lambda(t)^\trp \eta_1(t,s)\dt = -\int_s^T\lambda(t)^\trp F^k_x(t)\dt.
\end{equation}
Analogously, we get for $\eta_0(t,s) = -\chi_{[-\tau,t)}(s)\cdot   j^k_x(t)$
\begin{equation}\label{eq:eta0}
 \int_s^T \lambda_0 \eta_0(t,s)\dt = - \int_s^T \eta_0(t,s)\dt = \int_s^T j^k_x(t)\dt.
\end{equation}
Second, we find by elementary calculations
\[\begin{split}
 F^k_y(t)\lie_k'(x_{k,t}) \xi_t &=  F^k_y(t)
		\frac{\int_{t-\tau}^t \exp(k \, x_k(s)) \odot  \xi(s) \, \ds}
		{\int_{t-\tau}^t \exp(k \, x_k(\hat s)) \, \d\hat s}
		\\
		&= F^k_y(t)
		\int_{-\tau}^t \chi_{(t-\tau,t)}(s)
		\diag{
		\frac{\exp(k \, x_k(s))}
		{\int_{t-\tau}^t \exp(k \, x_k(\hat s)) \, \d\hat s}
		}
		\xi(s) \, \ds
		.
\end{split}\]
With the choice
\[
 \eta_2(t,s):= -F^k_y(t)
		\int_s^t \chi_{(t-\tau,t)}(s')
		\diag{
		\frac{\exp(k \, x_k(s'))}
		{\int_{t-\tau}^t \exp(k \, x_k(\hat s)) \, \d\hat s}
		}  \ds'
\]
we get the identity
\[
 F^k_y(t)\lie_k'(x_{k,t}) \xi_t =\int_{-\tau}^t  \d_s\eta_2(t,s)\xi(s).
\]
In addition, we find
\begin{equation}
\label{eq:eta2}
\begin{split}
\int_s^T \lambda(t)^\trp \eta_2(t,s)\dt
&=- \int_s^T \lambda(t)^\trp F^k_y(t)
		\int_s^t \chi_{(t-\tau,t)}(s')
		\diag{
		\frac{\exp(k \, x_k(s'))}
		{\int_{t-\tau}^t \exp(k \, x_k(\hat s)) \, \d\hat s}
		}   \ds' \,
\dt
\\
&=- \int_s^T \int_{s'}^{\min(s'+\tau,T)} \lambda(t)^\trp F^k_y(t)
		\diag{
		\frac{\exp(k \, x_k(s'))}
		{\int_{t-\tau}^t \exp(k \, x_k(\hat s)) \, \d\hat s}
		}
		\dt \, \ds'.
\end{split}
\end{equation}
Using $\eta:=\eta_1+\eta_2$, \eqref{eq:eta1}--\eqref{eq:eta2} in the adjoint equation \eqref{eq523} yields
\begin{multline*}
  \lambda(s)^\trp-\int_s^T\lambda(t)^\trp  F^k_x(t)\dt
  + \int_s^T j^k_x(t)\dt
  \\
  - \int_s^T \int_{s'}^{\min(s'+\tau,T)} \lambda(t)^\trp F^k_y(t)
		\diag{
		\frac{\exp(k \, x_k(s'))}
		{\int_{t-\tau}^t \exp(k \, x_k(\hat s)) \, \d\hat s}
		}
		\dt\,\ds'
  =0.
\end{multline*}
Since $\lambda$ is of bounded variation, the integrands are bounded functions, which implies
that $\lambda\in W^{1,\infty}(I;\R^n)$.
In addition, the differential equation
\begin{equation*}
  \lambda'(s)^\trp +\lambda(s)^\trp  F^k_x(s)
+\int_{s}^{\min(s+\tau,T)} \lambda(t)^\trp F^k_y(t)
		\diag{
		\frac{\exp(k \, x_k(s))}
		{\int_{t-\tau}^t \exp(k \, x_k(\hat s)) \, \d\hat s}
		}
		\dt
  =j^k_x(s)
\end{equation*}
is satisfied for almost all $s\in I$ with the terminal value $\lambda(T) =0$.
In addition, the result of \cite{Banks69} includes the maximum principle
\begin{multline*}
 \int_0^T \lambda_k(t)^\trp F(x_k(t),\lie_k(x_{k,t}),u_k(t)) - j(t,x_k(t),u_k(t))-  \frac12 \, \abs{u_k(t)-u^*(t)}^2\, \dt \ge
 \\
 \int_0^T \lambda_k(t)^\trp F(x_k(t),\lie_k(x_{k,t}),u(t)) - j(t,x_k(t),u(t))-  \frac12 \, \abs{u(t)-u^*(t)}^2\, \dt
\end{multline*}
for admissible $u$.
\added{
Hence, $u_k$ is the global minimum of the function
\[
 g(u):= \int_0^T -\lambda_k(t)^\trp F(x_k(t),\lie_k(x_{k,t}),u(t)) + j(t,x_k(t),u(t))\dt + \frac12 \|u(t)-u^*(t)\|_{L^2(I;\R^m)} ^2.
\]
As the mapping
 $u\mapsto\int_0^T -\lambda_k^\trp F(x_k,\lie_k(x_{k,t}),u) + j(t,x_k,u) \dt $
 is convex, the function $g$ is the sum of a convex function and the differentiable function $u\mapsto \frac12 \|u(t)-u^*(t)\|_{L^2(I;\R^m)} ^2$.
The necessary condition for $u_k$ to be a minimum of $g$ is given by \eqref{eq:maximum-principle}, see, e.g., \cite[Prop.\@ II.2.2]{EkelandTemam1999}.
}
\end{proof}
\added{
\begin{remark}
 \label{rem:on_regular_point}
 Let us give some further comments concerning the regular point condition
 which is utilized in \cite[Theorem~1]{Banks69},
 see also \cite[Theorem~4]{Banks1968}.
 This condition means that
 $t = T$ is a Lebesgue point of the functions
 \begin{align*}
  &t \mapsto j(t, x_k(t), u_k(t)) + \frac12 \, \abs{ u_k(t) - u^*(t)}^2
  ,
  \\
  &t \mapsto F(x_k(t), \lie_k(x_{k,t}), u_k(t))
  .
 \end{align*}
 By tracing the arguments in the proof, one can check that this condition
 is used only to derive a transversality condition
 in case of a free end time or terminal constraints.

 Alternatively, one could discuss a modified problem
 over a prolonged time horizon in which
 both the objective and the right-hand side of the state equation
 are extended by $0$.
 This means that for some $\hat T > 0$, we consider a problem over $[0,T+\hat T]$,
 for which the integrand in the objective is given by
 \begin{equation*}
  t \mapsto
  \begin{cases}
   j(t, x(t), u(t)) + \frac12 \, \abs{ u(t) - u^*(t)}^2
   &\text{if } t \in I \\
   0 & \text{else}
  \end{cases}
 \end{equation*}
 and with the right-hand side
 \begin{equation*}
  t \mapsto
  \begin{cases}
   F(x(t), \lie_k(x_t), u(t))
   &\text{if } t \in I \\
   0 & \text{else}
  \end{cases}
 \end{equation*}
 in the state equation.
 It is clear that this modified problem possesses a solution
 which is a constant extension of $(x_k, u_k)$
 and, therefore, the Lebesgue point property holds.
\end{remark}
}

Testing the adjoint equation \eqref{eq:adjoint_regularized} with a test function $v\in L^\infty(I_\tau;\R^n)$
with $v(t) = 0$ for $t \in (-\tau, 0)$
and undoing the interchanging
of integration order in the previous proof, we arrive at the following weak formulation of the
adjoint equation,
\begin{multline}\label{eq:adjoint_weak}
 \int_0^T
  \lambda_k'(t)^\trp v(t) +\lambda_k(t)^\trp  F^k_x(t)v(t)
  +
  \lambda_k(t)^\trp F^k_y(t)
		\frac{\int_{t-\tau}^t \exp(k \, x_k(s)) \odot v(s) \, \ds}
		{\int_{t-\tau}^t \exp(k \, x_k(\hat s)) \, \d\hat s}
  \dt
  \\
  =
  \int_0^T j^k_x(t)v(t)\dt
  ,
\end{multline}
which is more suitable for studying the limit $k\to\infty$.

\begin{theorem}
\label{thm:boundedness_adjoints}
The sequence $(\lambda_k)$ is bounded in $L^\infty(I;\R^n)$ and $W^{1,1}(I;\R^n)$.
\end{theorem}
\begin{proof}
 Since $(x_k)$ and $(u_k)$ are bounded in $C(I;\R^n)$ and $L^\infty(I;\R^m)$, respectively, we find that $(F^k_x)$, $(F^k_y)$, and $(j^k_x)$
 are bounded in $L^\infty(I;\R^{n\added{\times}n})$ and $L^\infty(I;\R^n)$, respectively.
 \added{
 Let $t'\in I$.
 }
 Setting $v(s) = \chi_{(t',T)}(s)\frac{\lambda_k(s)}{|\lambda_k(s)|}$ in \eqref{eq:adjoint_weak}, we obtain
 \[\begin{split}
  |\lambda_k(t')| &\le C \int_{t'}^T |\lambda_k(t)| + |\lambda_k(t)| \cdot \left| \frac{\int_{\max(t-\tau,t')}^t \exp(k \, x_k(s)) \odot \lambda_k(s)\cdot |\lambda_k(s)|^{-1} \, \ds}
		{\int_{t-\tau}^t \exp(k \, x_k(\hat s)) \, \d\hat s} \right| + 1\, \dt
		\\
		&\le  C \int_{t'}^T |\lambda_k(t)| + 1\, \dt
		\end{split}
 \]
with some constant $C$ independent of $k$.
By Gronwall's inequality
\added{
in integral form,
}
we find that $(\lambda_k)$ is bounded in $L^\infty(I;\R^n)$.

Now let $v\in L^\infty(I;\R^n)$ be arbitrary. Then we get from \eqref{eq:adjoint_weak}
the estimate
\[\begin{split}
 \int_0^T
  \lambda_k'(t)^\trp v(t) \dt
  &\le
  C\int_0^T |\lambda_k(t)| \cdot |v(t)|
  +
  ( |\lambda_k(t)| +1)\cdot \|v\|_{L^\infty(I;\R^n)}\dt
  \\ &
  \le C (\|\lambda_k\|_{L^\infty(I;\R^n)}+1)\|v\|_{L^\infty(I;\R^n)}
  \end{split}
\]
with some constant $C>0$ independent of $k$.
\added{By choosing $v = \operatorname{sign}(\lambda_k')$, we obtain the desired boundedness
of $(\lambda_k')$ in $L^1(I;\R^n)$.}
\end{proof}

\begin{lemma}
 \label{lem:liepk}
 The sequence of functions $(t,s)\mapsto \lie_k'(x_{k,t})(s)$ is uniformly bounded in $L^\infty(I;L^1(I_\tau))^n$, where
 we used the notation
\begin{equation*}
 \lie_k'(x_t)(s)
	=\frac{\chi_{(t-\tau,t)}(s) \exp\parens[\big]{k \, x(s)} }
	{\int_{t-\tau}^t \exp\parens[\big]{k \, x(\hat s)} \, \d\hat s}.
\end{equation*}
\end{lemma}
\begin{proof}
 Obviously, this function is non-negative.
 Let $t\in I$ be given. Then it holds
 \[
  \norm[\big]{\lie_k'(x_t)}_{L^1(I_\tau)}
  =
  \int_{-\tau}^T
  \lie_k'(x_t)(s) \ds
  =
  \int_{-\tau}^T \frac{\chi_{(t-\tau,t)}(s) \exp\parens[\big]{k \, x(s)} }
  {\int_{t-\tau}^t \exp\parens[\big]{k \, x(\hat s)} \, \d\hat s} \ds
  = e.
 \]
 Here, the $e$ is the vector in $\R^n$ with all entries equal to $1$.
\end{proof}

\section{Passing to the limit in the optimality system}
\label{sec:limit}

\added{
Again, let $(x^*,u^*)$ be a local solution of the original problem.
Let $(x_k,u_k)$ be a sequence of global solutions of the regularized optimal control problem \eqref{eq:constraint-regularized}.
By \cref{thm:convergence_regularized_optimal_controls}, we have
the convergence $x_k\to x^*$ and $u_k\to u^*$ in $C(I;\R^n)$ and $L^2(I;\R^m)$, respectively.
Then \cref{thm:optimality_condition_regularized} gives a first-order necessary optimality condition.
}

In this section,
we are going to pass to the limit $k\to \infty$ in the optimality system provided by \cref{thm:optimality_condition_regularized}.
The main work is to understand the behaviour
of the expression $\lie_k'(x_{k,t})(s)$
which appears in the adjoint equation \eqref{eq:adjoint_weak}.
We define $\mu_k \in L^\infty(I;L^1(I_\tau))^n$ via
\begin{equation}
 \label{eq:def_mu_k}
 \mu_k(t,s) := \lie_k'(x_{k,t})(s),
\end{equation}
cf.\ \cref{lem:liepk}.
We have seen in \cref{lem:liepk}
that $\mu_k$ is bounded in the space
$L^\infty(I; L^1(I))^n$.
Since this space is neither a dual space nor a reflexive space,
we cannot extract a subsequence which is weak($\star$)ly convergent
in this space.
Therefore, we embed this space into a suitable dual space,
see \cref{subsec:dualSpaceL1}.
Finally, \cref{subsec:opt_con}
contains the necessary optimality system.

\subsection{The dual space of \texorpdfstring{$L^1(I;C(I_\tau))$}{L1(I;C(I\_tau))}}
\label{subsec:dualSpaceL1}
It is well known that $L^1(I_\tau)$ is not a dual space.
Therefore, $L^\infty(I; L^1(I_\tau))$ cannot be a dual space as well.
A typical remedy is to embed $L^1(I_\tau)$ into $\MM(I_\tau)$,
where $\MM(I_\tau)$ is the space of regular signed Borel measures on $I_\tau$
equipped with the total variation norm,
since this is the dual space of $C(I_\tau)$,
see \cite[Theorem~6.19]{Rudin1987}.
We will see below,
that the dual space of $L^1(I;C(I_\tau))$ will be useful
in our situation.
\added{For $1 \le p < \infty$ and $p^{-1} + q^{-1} = 1$,
the dual space of $L^p(I;X)$ is $L^q(I;X\dualspace)$
if and only if
$X\dualspace$ possesses the
Radon-Nikodym property (RNP), see \cite[Theorem~IV.1]{DiestelUhl1977}.
However, $\MM(I_\tau)$ does not have the RNP,
see \cite[Section~VII.7]{DiestelUhl1977}.
Therefore, $L^1(I;C(I_\tau))\dualspace \ne L^\infty(I;\MM(I_\tau))$.}
In order to characterize the dual space,
we have to introduce a space of weak-$\star$ measurable functions.
We follow the presentation in \cite[Section~10.1]{PapageorgiouKyritsi-Yiallourou2009}.
\begin{definition}
 \label{def:linftymu}
 A function $\nu : I \to \MM(I_\tau)$ is said to be weak-$\star$ measurable
 if the function
 \begin{equation*}
  t \mapsto \dual{\nu(t)}{z}
 \end{equation*}
 is measurable for all $z \in C(I_\tau)$.
 If $\nu_1, \nu_2$ are weak-$\star$ measurable,
 we define the equivalence relation $\sim$
 via
 \begin{equation*}
  \nu_1 \sim \nu_2
  \quad\Leftrightarrow\quad
  \dual{\nu_1(t)-\nu_2(t)}{z} = 0
  \text{ f.a.a.\ $t \in I$ for all $z \in C(I_\tau)$}.
 \end{equation*}
 \added{By this definition,} the null set may depend on $z$.
 \added{%
  However, since $C(I_\tau)$ is separable, the null set can be chosen
  independently of $z \in C(I_\tau)$.
 }
 Finally, the space $L^\infty_{\added{w}}(I; \MM(I_\tau))$
 consists of all equivalence classes $[\nu]$ of weak-$\star$ measurable functions $\nu : I \to \MM(I_\tau)$
 satisfying
 \begin{equation*}
  \abs[\big]{\dual{\nu(t)}{z}} \le c \, \norm{z}_{C(I_\tau)}
  \qquad
  \text{f.a.a.\ $t \in I$ for all $z \in C(I_\tau)$}.
 \end{equation*}
 for some $c \ge 0$.
 Again, the null set \added{can be chosen independently of} $z\in C(I_\tau)$.
 The infimum of all these constants $c \ge 0$ is denoted by $\norm{\nu}_{L^\infty_{\added{w}}(0,T;\MM(I_\tau))}$.
 This is a norm on $L^\infty_{\added{w}}(0,T;\MM(I_\tau))$.
\end{definition}
\added{Note that a Bochner-measurable function from $I$ to $L^1(I_\tau)$ is also weak-$\star$ measurable as a function from $I$ to $\MM(I_\tau)$.}
\begin{theorem}
 \label{thm:dual_of_l1c}
 The dual space of $L^1(0,T;C(I_\tau))$
 is isometrically isomorphic to the space $L^\infty_{\added{w}}(I;\MM(I_\tau))$ via the duality pairing
 \begin{equation*}
  \dual{\nu}{z}
  :=
  \int_0^T \dual{\nu(t)}{z(t)} \, \dt
 \end{equation*}
 for all $\nu \in L^\infty_{\added{w}}(I;\MM(I_\tau))$ and $z \in L^1(I;C(I_\tau))$.
\end{theorem}
This result can be found in
\cite[Theorem~10.1.16]{PapageorgiouKyritsi-Yiallourou2009},
see also
\cite[Theorem~8.18.2]{Edwards1995}.
We also note that the measurability of
\begin{equation*}
 t \mapsto \dual{\nu(t)}{z(t)}
\end{equation*}
(which is essential for the definition of the above duality pairing)
is proven in the first part of the proof of
\cite[Theorem~8.18.2]{Edwards1995}.

In the sequel, it will be useful to define
a function $f \otimes y \in L^1(I;C(I_\tau))$ for $f \in L^1(I)$
and $y \in C(I_\tau)$
via
\begin{equation*}
 (f \otimes y)(t,s) := f(t) \, y(s)
\end{equation*}
for all $t \in I$ and $s \in I_\tau$.
In addition, we need the following notation
for the component-wise application of $\mu \in \MM(I_\tau)^n$ to
$v\in C(I_\tau)^n$:
\begin{equation}\label{eq:def_odot}
 \langle \mu \odot v \rangle := \big(\, \langle \mu_i,v_i\rangle \, \big)_{i=1,\ldots,n}.
\end{equation}
An analogue notation is used for
$\mu \in L^\infty_{\added{w}}(I; \MM(I_\tau))^n$ and $v \in L^1(I; C(I_\tau))^n$.

Now we are faced with the following situation:
$\mu_k$ is a bounded sequence in the space $L^\infty(I;L^1(I_\tau))^n$
(\cref{lem:liepk})
and
this space is isometrically embedded into the dual space
$L^\infty_{\added{w}}(I;\MM(I_\tau))^n = (L^1(I;C(I_\tau))^n)\dualspace$.
This leads to the following result.
\begin{lemma}
 \label{lem:convergence}
 We can extract a subsequence of $(\mu_k)$
 (without relabeling) such that
 \begin{equation*}
  \mu_k \weaklystar \mu
  \quad\text{in }
  (L^1(I; C(I_\tau))^n)\dualspace
  =
  L^\infty_{\added{w}}(I; \MM(I_\tau))^n.
 \end{equation*}
 The limit $\mu$
 satisfies
 \begin{equation*}
  \mu(t)
  \in
  \partial \max x_t
 \end{equation*}
 for a.a.\ $t \in I$.
 Here, $\partial \max x_t$
 is the (coefficientwise) (convex) subdifferential of the function
 $C(I_\tau)^n \ni x \mapsto \max x_t$ at the optimal state $x$.
\end{lemma}
\begin{proof}
 The first claim follows from the
 Banach-Alaoglu-Bourbaki theorem.

 The definition \eqref{eq:def_mu_k} of $\mu_k$ implies
 \begin{equation*}
  \mu_k(t, \cdot) \in \partial \lie_k(x_{k,t}).
 \end{equation*}
 Here, $\partial \lie_k(x_{k,t})$ is the coefficientwise
 convex subdifferential of $C(I_\tau)^n \ni x_k \mapsto \lie_k(x_{k,t})$ at $x_k$.
 Thus, for arbitrary $\varphi \in L^\infty(I)$, $\varphi \ge 0$
 and $z \in C(I_\tau)^n$,
 we have
 \begin{equation*}
  \lie_k(z_t) \ge \lie_k(x_{k,t}) + \int_{-\tau}^T \mu_k(t, s) \odot z(s) \, \ds,
 \end{equation*}
 thus,
 \begin{align*}
  \int_0^T \varphi(t) \, \lie_k(z_t) \, \dt
  &\ge
  \int_0^T \varphi(t) \, \lie_k(x_{k,t}) \, \dt
  +
  \int_0^T \int_{-\tau}^T \varphi(t) \, \mu_k(t, s) \odot z(s) \, \ds \, \dt
  \\
  &\ge
  \int_0^T \varphi(t) \, \lie_k(x_{k,t}) \, \dt
  +
  \dualo{\mu_k}{(\varphi \otimes z)}
  .
 \end{align*}
 By using \cref{lem:l1norm} and the Lipschitz continuity
 of $\lie_k$,
 we can pass to the limit $k \to \infty$.
 This yields
 \begin{align*}
  \int_0^T \varphi(t) \, \max z_t \, \dt
  &\ge
  \int_0^T \varphi(t) \, \max x_t^{\added{*}} \, \dt
  +
  \dualo{\mu}{(\varphi \otimes z)}
  \\
  &=
  \int_0^T \varphi(t) \, \max x_t^{\added{*}} \, \dt
  +
  \int_0^T \varphi(t) \, \dualo{\mu(t)}{z} \, \dt
  .
 \end{align*}
 Since $\varphi \ge 0$ is arbitrary, this yields
 \begin{equation*}
  \max z_t
  \ge
  \max x_t^{\added{*}}
  +
  \dualo{\mu(t)}{z}
 \end{equation*}
 for a.a.\ $t \in I$.
 Note that the null set may depend on $z \in C(I_\tau)^n$.
 By using the separability of $C(I_\tau)^n$,
 we can show that the null set can be chosen independently of $z$.
 Thus,
 \begin{equation*}
  \max z_t
  \ge
  \max x_t^{\added{*}}
  +
  \dualo{\mu(t)}{z}
  \qquad \forall z \in C(I_\tau)^n
 \end{equation*}
 holds for a.a.\ $t \in I$.
 This shows the claim.
\end{proof}
The standard characterization of the subdifferential of
the maximum function yields the following properties of $\mu$.
\begin{corollary}
 \label{cor:mu}
 The limit $\mu$ from \cref{lem:convergence}
 satisfies $\mu \ge 0$,
 \begin{equation*}
  \norm{\mu(t)_i}_{\MM(I_\tau)}
  =
  1
  , \qquad \forall i =1,\ldots,n,
 \end{equation*}
 and
 \begin{equation*}
  \supp(\mu(t)_i)
  \subset
  \argmax_{s \in [t-\tau,t]} x_i(s),
  \qquad i=1,\ldots,n
 \end{equation*}
 for almost all $t \in I$.
\end{corollary}
Here, $\supp$ denotes the support of a measure.
Thus,
$\mu(t)_i$
is a probability measure
supported at the maximizers of $x$ on the interval $[t-\tau,t]$.

\subsection{Necessary optimality conditions}
\label{subsec:opt_con}
For abbreviation, let us define
\[
 F^*(t):= F(x^*(t),\max x^*_t,u^*(t)).
\]
Similarly, we define  $F^*_x(t)$, $F^*_y(t)$, $j^*_x$ to denote the derivatives of $F$ and $j$ with respect to the first and second argument, respectively,
evaluated along the optimal state and control.

\begin{theorem}
 \label{thm:opt_sys}
 Let  $(x^*,u^*)$ be a local solution of the original problem.
 Then there \added{are} $\lambda\in BV(I;\R^n)$ and $\mu\in  L^\infty_{\added{w}}(I; \MM(I_\tau))^n$ such that the following optimality system is satisfied:
\begin{enumerate}
 \item (Adjoint equation) For all $v\in C(I_\tau;\R^n)$ with $v(s)=0$ for all $s\in[-\tau,0]$ it holds
\begin{equation}\label{eq:nonsmooth_adjoint}
 \int_0^T \d\lambda(t)^\trp v(t)
 +\int_0^T\lambda(t)^\trp  F^*_x(t)v(t)
  +
  \lambda(t)^\trp F^*_y(t) \langle\mu(t) \odot v\rangle
  \dt
  =
    \int_0^T j^*_x(t)v(t)\dt
.
\end{equation}
Here, $\langle\mu(t)\odot v\rangle$ denotes the vector with entries $\langle \mu_i(t),v_i\rangle$, see \eqref{eq:def_odot}.
\item (Maximum principle)
The inequality
\begin{multline}\label{eq:nonsmooth-maximum-principle}
 \int_0^T \lambda(t)^\trp F(x^*(t),\max x^*_t,u^*(t)) - j(t,x^*(t),u^*(t))\dt \ge
 \\
 \int_0^T \lambda(t)^\trp F(x^*(t),\max x^*_t,u(t)) - j(t,x^*(t),u(t))\dt
\end{multline}
holds for all feasible controls $u$.

\item (Subdifferential condition)
The measure-valued function $\mu$ satisfies
 \begin{equation*}
  \mu(t)
  \in
  \partial \max x^*_t
 \end{equation*}
 for almost all $t\in I$.
 \end{enumerate}
 \end{theorem}
\begin{proof}
 Let $(x_k,u_k)$ be a sequence of global solutions of the regularized optimal control problem as considered in \cref{sec52}.
 Then $x_k\to x^*$ in $C(I;\R^n)$ and $u_k\to u^{\added{*}}\in L^2(I;\R^m)$.
 For sufficiently large $k$ the requirements of \cref{thm:optimality_condition_regularized} are satisfied. Hence,
 there is a sequence $(\lambda_k)$ in $W^{1,\infty}(I;\R^n)$ such that
 \eqref{eq:adjoint_regularized}--\eqref{eq:maximum-principle} is satisfied.
 In addition, the weak formulation of the adjoint equation \eqref{eq:adjoint_weak} holds.
 By \cref{thm:boundedness_adjoints}, the sequence $(\lambda_k)$ is bounded in $W^{1,1}(I;\R^n)$.
 \added{We recall the definition of $\mu_k$ from \eqref{eq:def_mu_k},}
\begin{equation*}
 \mu_k(t,s) = \lie_k'(x_{k,t})(s), s\in I_\tau, \ t\in I.
\end{equation*}
Then by \cref{lem:liepk}, the sequence $(\mu_k)$ is bounded in $L^\infty(I;L^1(I_\tau))^n$.
Using
\added{the compact embedding $W^{1,1}(I; \R^n) \hookrightarrow L^p(I; \R^n)$, $p < \infty$,
the Banach-Alaoglu theorem in $\MM(I;\R^n) = C(I;\R^n)\dualspace$,}
and \cref{lem:convergence},
there is (after extracting subsequences if necessary) $\lambda\in BV(I;\R^n)$ and $\mu$ such that
\begin{alignat}{2}
 \lambda_k &\to \lambda &&\text{ in } L^p(I;\R^n) \quad \forall p<\infty, \\
 \lambda_k' &\weaklystar \lambda' &&\text{ in } \MM(I;\R^n) = C(I;\R^n)\dualspace, \\
 \mu_k &\weaklystar \mu &&\text{ in }
  L^1(I; C(I_\tau))\dualspace
  =
  L^\infty_{\added{w}}(I; \MM(I_\tau)).
\end{alignat}
Passing to the limit in the maximum principle \eqref{eq:maximum-principle} to get \eqref{eq:nonsmooth-maximum-principle} is straightforward.
In the weak formulation of the adjoint equation \eqref{eq:adjoint_weak}, let us argue the convergence of the third term.
To this end, let $v\in C(I_\tau;\R^n)$ be given with $v(s)=0$ for all $s\in[-\tau,0]$.
Due to the definition of $\mu_k$, we have
\[
 \int_0^T
  \lambda_k(t)^\trp F^k_y(t)
		\frac{\int_{t-\tau}^t \exp(k \, x_k(s)) \odot v(s) \, \ds}
		{\int_{t-\tau}^t \exp(k \, x_k(\hat s)) \, \d\hat s}
  \dt
  =
 \int_0^T
  \lambda_k(t)^\trp F^k_y(t)
		\langle \mu_k(t) \odot v\rangle
  \dt.
\]
By the convergence properties above, we have that the functions $t\mapsto \langle \mu_k(t) \odot v\rangle$ converge weak-$\star$ in
$L^\infty(I\added{;\R^n})$ to $t\mapsto \langle \mu(t) \odot  v\rangle$. This allows the passage to the limit in the adjoint equation to obtain  \eqref{eq:nonsmooth_adjoint}.
The subdifferential condition is a consequence \added{of} \cref{lem:convergence}.
\end{proof}

\added{
\begin{remark}
 It is well known that in the limiting process to obtain the optimality system of the original system certain information are lost.
 When inspecting the system provided by \cref{thm:opt_sys}, the only place with possible ambiguity is the inclusion
 $  \mu(t)  \in  \partial \max x^*_t$ for such $t$ where the maximum of $x^*_{i,t}$ is attained at more than one place.
 Pretending that the map $u\mapsto x$ is directional differentiable and following the arguments in \cite{ChristofMeyerWaltherClason2018}
 one obtains  the condition
 \[
  [ F_y^*(t)^\trp \lambda(t) ]_i \le 0
\]
for all $i$ and $t$ such that the maximum of $x_{i,t}$ is not unique.
This condition is not included in the result of \cref{thm:opt_sys}.
A rigorous derivation of such a condition is beyond the scope of this paper and is subject of future work.
\end{remark}
}

\subsection{Continuity properties of the adjoint}
\label{subsec:jump_adoint}

In this subsection, we analyze the continuity of the adjoint state $\lambda$.
Our first result  gives an expression for the difference between limits from the left and right.
Recall that $\lambda$ has bounded variation
and this ensures the existence of these one-sided limits.

\begin{theorem}\label{lem66}
 Let $s_0\in (0,T)$ be given.
 Then it holds
 \[
 \lambda_i(s_0+) - \lambda_i(s_0-) = \int_0^T \lambda(t)^\trp F_y^*(t) e_i \cdot \mu_i(t)(\{s_0\}) \,\dt,
 \added{
 \quad i=1,\ldots,n
 }
 ,
 \]
 where $ \lambda_i(s_0+)$ and $\lambda_i(s_0-) $ denote the limits from the right and from the left of $\lambda_i$ at $s_0$, respectively,
 \added{
 and where $e_i$ is the $i$-th unit vector.
 }
\end{theorem}
\begin{proof}Let $(s_1^j)$, $(s_2^j)$, $(\epsilon_j)$ be sequences such that $s_1^j,s_2^j\in I$ for all $j$, $s_1^j \nearrow s_0$,
 $s_2^j \searrow s_0$, and $\epsilon_j\searrow0$.
 Let $(v_j)$ be the sequence of piecewise linear functions in $C(I)$ defined by $\supp v_j =[s_1^j-\epsilon_j,s_2^j+\epsilon_j]$
 and $v_j(s_1^j)=v_j(s_2^j)=1$ for all $j$, and
 with kinks in $s_1^j-\epsilon_j$, $s_1^j$, $s_2^j$ and $s_2^j+\epsilon_j$.
Testing \eqref{eq:nonsmooth_adjoint} with $v_j \cdot e_i$,
 yields
\[
\int\limits_{0}^T\d\lambda_i(t)v_j(t) =\int_0^T  \big( j_x^*(t)- \lambda(t)^\trp F_x^*(t) \big) e_i v_j(t) \, \dt
- \int_0^T \lambda(t)^\trp F_y^*(t) \langle \mu(t) \odot e_i v_j \rangle \, \dt.
\]
Here, the second integral vanishes for $j\to\added{\infty}$.
For the left-hand side we have
\begin{align*} \label{eq: lr-limit}
\int\limits_{0}^T\d\lambda_i(t)v_j(t) &= - \int_{0}^{T}\lambda_i(t) v_j'(t) \, \dt
= -\frac{1}{\eps_j} \int_{s_1^j-\eps_j}^{s_1^j} \lambda_i(t)\; \dt +
   \frac{1}{\eps_j} \int_{s_2^j}^{s_2^j+\eps_j} \lambda_i(t)\; \dt\\
&\overset{j \to \added{\infty}}{\longrightarrow} \lambda_i(s_0+) - \lambda_i(s_0-).
\end{align*}
It remains to prove convergence for the third integral. We will use \added{the} dominated convergence theorem. First, we have the integrable
 bound
 \[
 \left| \lambda(t)^\trp F_y^*(t) \langle \mu(t) \odot e_i v_j \rangle \right|
 \le C \|\mu(t)\|_{\mathcal{M}(I_\tau)}
 \]
 since \deleted{$\lambda^*$}\added{$\lambda$} and $F_y^*$ are bounded, and $\|v_j\|_{C(I)}\le 1$. In addition,
 \[
 \mu_i(t)(\{s_0\}) \le \langle \mu_i(t), v_j\rangle \le \mu_i(t)([s_1^j-\epsilon,s_2^j+\epsilon])
 \]
 by the non-negativity of $\mu(t)$ and $\chi_{\{s_0\}}\le v_j \le \chi_{[s_1^j-\epsilon,s_2^j+\epsilon]}$. This proves the pointwise convergence $\langle \mu_i(t), v_j\rangle \to  \mu_i(t)(\{s_0\})$.
 The dominated convergence theorem yields
 \[
  \lim_{j\to\infty}
 \int_0^T \lambda(t)^\trp F_y^*(t) \langle \mu(t) \odot e_i v_j \rangle \,\dt = \int_0^T \lambda(t)^\trp F_y^*(t)e_i \mu_i(t)(\{s_0\}) \,\dt,
 \]
and the claim is proven.
\end{proof}

This shows that $\lambda_i$ can be discontinuous only for those $s_0$, for which $\mu_i(t)(\{s_0\})$
is non-zero on a set of positive measure.
Due to the subdifferential condition in \cref{thm:opt_sys},
this is equivalent to the statement that $\max x_{i,t}^*= x_i^*(s_0)$
for $t$ in a set of positive measure.

\begin{corollary}
  \label{lem:jumps}
  \added{
  Let $i\in \{1,\ldots,n\}$.
  }
	\begin{enumerate}
		\item Let $s_0 \in (0,T)$ be such that $s_0 \notin \operatorname{argmax}x_{i,t}^*$ for almost all $t\in I$. Then $\lambda_i$ is continuous in $s_0$.
		\item Let $s_0 \in (0,T)$ be not a local maximum of $x_i^*$. Then $\lambda_i$ is continuous in $s_0$.
		\item Let $s_0 \in I$ be a strict local maximum of $x_i^*$.
		Assume that there exists a closed interval $[s_1, s_2] \subset [s_0,s_0+\tau]$
		such that $s_0$ is the unique maximum of $x^*_{i,t}$ for all $t\in [s_1,s_2]$,
		and $\max x^*_{i,t} > x_i^*(s_0)$ for all $t\in [s_0,s_0+\tau] \setminus [s_1,s_2]$,
		then it holds
		\begin{equation*}
		\lambda_i(s_0-) - \lambda_i(s_0+) = - \int_{s_1}^{s_2}  \lambda(t)^\trp F^*_y(t) e_i\,\dt.
		\end{equation*}
	\end{enumerate}
\end{corollary}
\begin{proof}
\added{
The proof relies on the inclusion $ \supp(\mu(t)_i)  \subset  \argmax_{s \in [t-\tau,t]} x_i(s)$,  $i=1,\ldots,n$, provided by \cref{cor:mu}.
}
	\begin{enumerate}
		\item
		By assumption, the support of $\mu_i(t)$ does not contain $s_0$ for almost all $t\in I$. Hence, it holds
		$\mu_i(t)(\{s_0\})=0$  for almost all $t\in I$,  and we obtain the continuity of $\lambda_i$ at $s_0$ by \cref{lem66}.

		\item
		Due to the assumption, there exists a sequence $(t^j)_j \subset I$ with $t^j \to s_0$ and $x_i^*(t^j)> x_i^*(s_0)$.
		Suppose that there is a subsequence such that (without relabelling) $t^j>s_0$ for all $j\in \N$.
		Then it follows $\max x_t > x(s_0)$ and $\mu_i(t)(\{s_0\})=0$ for all $t>s_0$, which implies $\mu(t)(\{s_0\})=0$ for all $t\ne s_0$.
		If there is a subsequence such that (without relabelling) $t^j<s_0$ for all $j\in \N$,
		then $\mu_i(t)\{s_0\}=0$ for all $t<s_0+\tau$, which implies $\mu(t)(\{s_0\})=0$ for all $t\ne s_0 + \tau$.
		Hence, the claim follows from \cref{lem66}.

		\item
		Due to the assumptions, it holds $\mu_i(t) =\delta_{s_0}$ for all $t\in [s_1,s_2]$.
		In addition, $\mu_i(t)(\{s_0\}) =0$ for all $t \not\in [s_1,s_2]$. The claim now follows from \cref{lem66}.

	\end{enumerate}
\end{proof}

\section{Numerical experiment}
\label{sec:numerics}
In this final section,
we present some numerical results for
the regularized optimal control problems
\begin{align*}
\min\quad  &\frac{\alpha}{2} \| x - x_d \|_{L^2(I)}^2 + \frac{\beta}{2} \| u \|_{L^2(I)}^2 \\
\text{s.t.}\quad &\dot{x}(t) = x(t) - 2\operatorname{LIE}_k(x_t) + u(t), \qquad t\in I,  \\
&x\vert_{[-\tau,0]} = 0,\\
&-5\leq u(t) \leq 5, \qquad t\in I.
\end{align*}
In particular,
we use the parameters
\begin{equation*}
\alpha = 100,\quad \beta = 0.1, \quad T= 3, \quad \tau = 0.2.
\end{equation*}
The desired state is given by
		\begin{equation*}
			x_d(t):=
			\begin{cases}
				\frac{1}{2} - \vert t- \frac{1}{2} \vert & 0\leq t \leq 1,\\
				0 & 1< t < 2,\\
				\vert t- \frac{5}{2} \vert - \frac{1}{2} & 2\leq t \leq 3=T.
			\end{cases}
		\end{equation*}
\added{
For this problem, the maximum principle \eqref{eq:nonsmooth-maximum-principle} is equivalent to the equation $\lambda - \beta u^*=0$.
}

\added{
State and adjoint equation were solved with the explicit Euler method with uniform step-size $\Delta t$.
In order to compute the $\lie$-operator applied to the discretized state $x^h$, we employed the following strategy.
Let $N$ such that $\tau = N \cdot \Delta t$.
First, the time-discretization leads to an quadrature rule to approximate the integrals
\begin{equation*}
 \lie_k(x^h_{t_j})
 =
 \frac1k \, \log\parens[\Big]{\int_0^\tau \exp\parens[\big]{k \, x^h_{t_j}(s)} \, \ds}
 \approx
 \frac1k \, \log\parens[\Big]{\Delta t\sum_{i=0}^{N-1}    \exp\parens[\big]{k \, x^h(t_{j-i})} }
 .
\end{equation*}
In order to avoid floating point overflow, we used the shifted version of LogSumExp.
Let us denote $x^h_{j,\max}:=\max_{i\in \{0,\ldots,N-1\}} x^h(t_{j-i})$. Then it holds
\[
  \frac1k\log\parens[\Big]{\Delta t\sum_{i=0}^{N-1}    \exp\parens[\big]{k \, x^h(t_{j-i})} }
 = x^h_{j,\max} + \frac1k \log\parens[\Big]{\Delta t\sum_{i=0}^{N-1}    \exp\parens[\big]{k \, (x^h(t_{j-i})- x^h_{j,\max}}) }.
\]
We refer to \cite{blanchard2019accurate} for a floating point rounding error analysis of this shifted LogSumExp computation.
The discrete adjoint equations were obtained by transposing the linearized time-discrete state equations.
}

The optimization problem is solved by the projected gradient method
using the Armijo step size rule.
As initial guess for the control, we choose the zero function. We present
numerical results for time discretization step-size
$\Delta t = 10^{-4}$ and regularization parameters $k = 10^6$.
The computed solutions of the regularized and discretized optimal control problem
can be found in \cref{fig:1}.
\begin{figure}[ht]
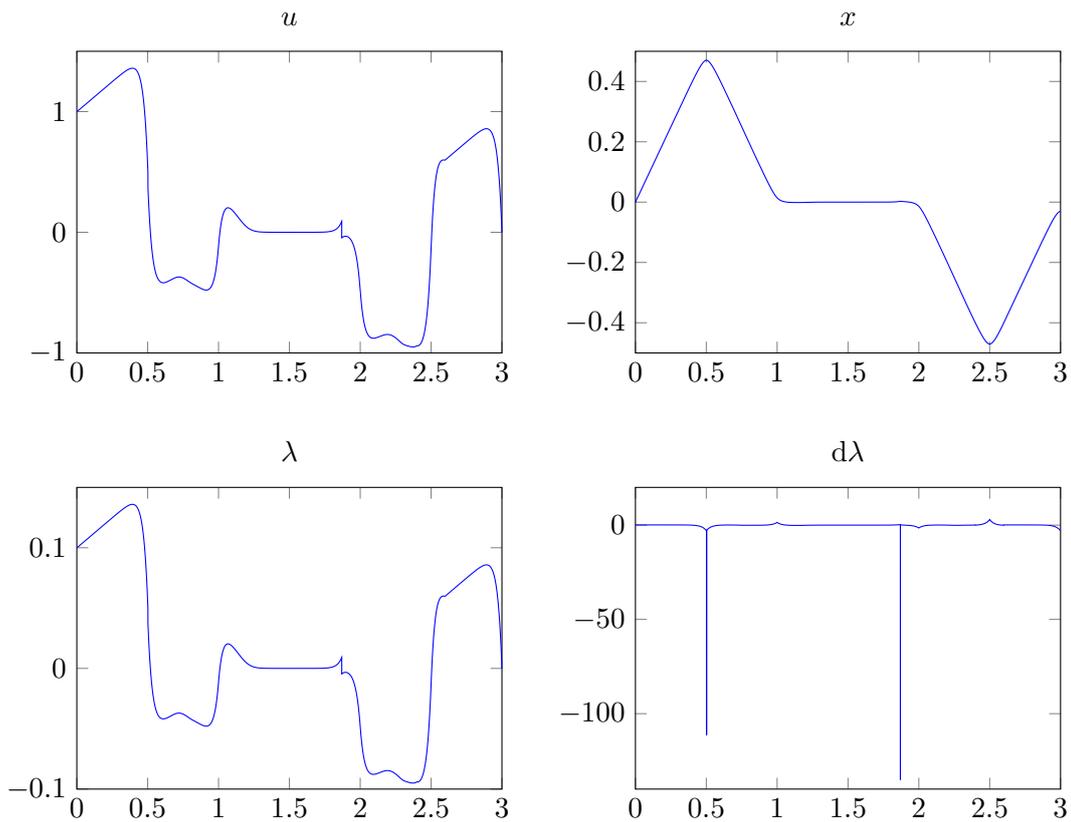

	\begin{minipage}[]{.5\textwidth}
		\hfill\input{plot_u.tikz}
	\end{minipage}%
	\begin{minipage}[]{.5\textwidth}
		\hfill
%
%
\begin{tikzpicture}

\begin{axis}[%
width=0.761\textwidth,
height=4cm,
at={(0\textwidth,0cm)},
scale only axis,
xmin=0,
xmax=3,
ymin=-0.5,
ymax=0.5,
axis background/.style={fill=white},
title style={font=\bfseries},
title={$x$},
xtick distance=.5
]
\addplot [color=blue, forget plot]
  table[row sep=crcr]{%
0	0\\
0.3426	0.342058930494952\\
0.3784	0.377404840601293\\
0.4004	0.398762132152779\\
0.4164	0.413930415400987\\
0.429	0.425511989609623\\
0.4394	0.4347095346711\\
0.4482	0.442135788866401\\
0.4559	0.448278649242522\\
0.4626	0.45327975043124\\
0.4686	0.457421754154956\\
0.474	0.46082063716545\\
0.4789	0.463584712751609\\
0.4834	0.465809662838523\\
0.4875	0.467535635593297\\
0.4913	0.468843125312794\\
0.4949	0.469789129743426\\
0.4983	0.4703896164073\\
0.5015	0.470666454656106\\
0.5029	0.470587350319651\\
0.5062	0.47003820664942\\
0.5097	0.469158694343558\\
0.5134	0.46793130690805\\
0.5174	0.466298717372641\\
0.5217	0.46423157458332\\
0.5264	0.461649645507279\\
0.5315	0.458519804604105\\
0.5371	0.454748440112569\\
0.5433	0.450230782369616\\
0.5503	0.444775491100569\\
0.5582	0.438256839685623\\
0.5673	0.430378215835214\\
0.578	0.420735188594885\\
0.5909	0.408721509041157\\
0.6068	0.393522176517874\\
0.6276	0.373240023407489\\
0.6556	0.345533235741021\\
0.6877	0.313389456586134\\
0.7154	0.285273910512733\\
0.825	0.173693488570469\\
0.8598	0.138701024716096\\
0.8907	0.10800319739088\\
0.9137	0.085516012134085\\
0.9309	0.0690669861075004\\
0.9443	0.0566158618956831\\
0.9553	0.0467567128262534\\
0.9646	0.0387813945195097\\
0.9726	0.0322752825644592\\
0.9796	0.0269290531767901\\
0.9858	0.0225312301860412\\
0.9914	0.0188908994049255\\
0.9965	0.0159029643671666\\
1.0011	0.0135221300768205\\
1.0057	0.011457481747926\\
1.0106	0.00956473761762622\\
1.0159	0.00782214582327212\\
1.0216	0.00624694947257742\\
1.0278	0.00482709135473458\\
1.0346	0.00355874447982485\\
1.0422	0.00242932003303453\\
1.0508	0.00143961942553883\\
1.0607	0.000589115068313717\\
1.0723	-0.000119244375960115\\
1.0863	-0.000686070459445354\\
1.1037	-0.00110317831977014\\
1.126	-0.00134919307153636\\
1.1538	-0.00137162717170014\\
1.1853	-0.00111857568372375\\
1.2766	-0.00016199292857344\\
1.3291	-3.09579728532405e-05\\
1.502	-3.96917525957718e-07\\
1.7747	0.000137215055693307\\
1.8109	0.000432241663410604\\
1.8331	0.00087295705581214\\
1.8493	0.00145770817664248\\
1.8621	0.00218537577412814\\
1.8714	0.00269757631332945\\
1.8922	0.00185795203438577\\
1.923	0.000594598335350582\\
1.937	-0.000267573833590884\\
1.948	-0.00121366298586967\\
1.9572	-0.00227465469980803\\
1.9652	-0.00346997072507493\\
1.9723	-0.00480724225538776\\
1.9787	-0.00629315668201302\\
1.9845	-0.00792191408026399\\
1.9898	-0.00969272552599154\\
1.9948	-0.0116565629876528\\
1.9995	-0.0138068283831156\\
2.004	-0.0161759726202848\\
2.0089	-0.0190747996486551\\
2.0143	-0.0225992321709612\\
2.0203	-0.0268570098930594\\
2.027	-0.0319631000691301\\
2.0346	-0.0381181158238801\\
2.0432	-0.0454494434684816\\
2.0532	-0.0543470945839419\\
2.0651	-0.0653177594724648\\
2.0798	-0.0792642997566673\\
2.0985	-0.0973998348468421\\
2.1229	-0.121464647571733\\
2.1539	-0.152436792049283\\
2.1882	-0.187093694506698\\
2.2434	-0.243356360379431\\
2.2915	-0.292681656732737\\
2.3307	-0.33283848361343\\
2.3549	-0.357205343454506\\
2.3761	-0.378174154644043\\
2.3934	-0.394915178501741\\
2.4093	-0.409908896760609\\
2.4238	-0.423248743678414\\
2.4349	-0.433099715020584\\
2.4441	-0.440907136835804\\
2.452	-0.447257423984881\\
2.4589	-0.452456552519822\\
2.465	-0.45671529508814\\
2.4705	-0.460224273906909\\
2.4755	-0.463089454995385\\
2.4801	-0.465404795186325\\
2.4843	-0.467208266205023\\
2.4882	-0.468579384133586\\
2.4918	-0.469551495748152\\
2.4952	-0.470179324042499\\
2.4984	-0.470483382517267\\
2.5014	-0.470489841426744\\
2.5045	-0.470211585561857\\
2.5078	-0.469625559248446\\
2.5113	-0.468709549620807\\
2.515	-0.467445837960104\\
2.519	-0.465775825411362\\
2.5233	-0.463669799742866\\
2.528	-0.46104597930261\\
2.5331	-0.457870002175919\\
2.5387	-0.454045910112575\\
2.5449	-0.449465595668104\\
2.5518	-0.444012709487464\\
2.5595	-0.437566049277947\\
2.5682	-0.429911805489157\\
2.578	-0.420912174276922\\
2.5889	-0.41052035010679\\
2.6007	-0.398885428499335\\
2.8544	-0.145759802805332\\
2.8843	-0.116370389706748\\
2.904	-0.0973698401692342\\
2.9188	-0.0834598467290459\\
2.9307	-0.0726418638302344\\
2.9406	-0.064006377017586\\
2.949	-0.0570351277847645\\
2.9563	-0.0513236932800356\\
2.9628	-0.0465809290366836\\
2.9686	-0.0426838635820186\\
2.9739	-0.0394551051557857\\
2.9787	-0.0368546883797496\\
2.9831	-0.0347858392089839\\
2.9871	-0.0332053581457901\\
2.9909	-0.0320050979519286\\
2.9944	-0.0311920261093004\\
2.9977	-0.0307125646878403\\
3	-0.030559742433407\\
};
\end{axis}
\end{tikzpicture}%
	\end{minipage}%
	\\[.5cm]
	\begin{minipage}[]{.5\textwidth}
		\hfill\input{plot_lambda.tikz}
	\end{minipage}%
	\begin{minipage}[]{.5\textwidth}
		\hfill\input{plot_dlambda.tikz}
	\end{minipage}%
	\caption{Plots for control $u$, state $x$, adjoint $\lambda$ and its time derivative $\d\lambda$ }
	\label{fig:1}
\end{figure}
As one can see from the plot of $\d\lambda$ in \cref{fig:1}, the adjoint variable $\lambda$ has jumps
at the strict local maxima of the optimal state $x$ at time points $t\approx 0.50$ and $t\approx 1.87$.

\ifbiber
	\printbibliography
\else
	\bibliographystyle{plain_abbrv}
	\bibliography{references}
\fi
\end{document}